\documentclass{amsproc}
\usepackage{amsmath,amssymb,amsthm,amsfonts,amsbsy, enumerate}
\usepackage[dvips]{epsfig}
\theoremstyle{plain}
\newtheorem{theorem}{Theorem}[section]
\newtheorem{corollary}[theorem]{Corollary}

\newtheorem{lemma}[theorem]{Lemma}

\newtheorem{claim}{Claim}
\newtheorem{definition}[theorem]{Definition}
\newtheorem{rk}{Remark}[section]

\newcommand{\tr}{{\rm tr}}

\numberwithin{equation}{section}

\begin{document}

\title{Mixing Rates of Random Walks with Little Backtracking}

\author{Sebastian M. Cioab\u{a}}
\address{Department of Mathematical Sciences, University of Delaware, Newark, DE 19716-2553, USA.}
\email{cioaba@udel.edu}
\thanks{The author Sebastian M. Cioab\u{a} was supported by the National Security Agency grant H98230-13-1-0267.}

\author{Peng Xu}
\address{Department of Mathematical Sciences, University of Delaware, Newark, DE 19716-2553, USA.}
\email{xpeng@udel.edu}
\thanks{The author Peng Xu was supported by the U.S. National Science
Foundation through grants DMS-1409504 and CCF-1346564.}

\subjclass{Primary 05C81, 05E30, 15A18; Secondary 60B10, 60C05, 	60G99, 60J10}
\date{March 25, 2015 and, in revised form, May 5, 2015.}


\keywords{Regular graph, cliques, random walk, mixing rate, eigenvalues}

\begin{abstract}
Many regular graphs admit a natural partition of their edge set into cliques of the same order such that each vertex is contained in the same number of cliques. In this paper, we study the mixing rate of certain random walks on such graphs and we generalize previous results of Alon, Benjamini, Lubetzky and Sodin regarding the mixing rates of non-backtracking random walks on regular graphs.
\end{abstract}

\maketitle

%
.

\section{Introduction}

Let $G=(V,E)$ be a connected, non-complete and non-bipartite graph. Assume that its edges are partitioned into a set of cliques $\mathcal{K}$ such that each clique in $\mathcal{K}$ has the same order $l$ and each vertex of $G$ is contained in precisely $d$ cliques from $\mathcal{K}$. Obviously, $|\mathcal{K}|=\frac{d|V|}{l}$ and $G$ is $d(l-1)$-regular. When $l=2$, this is equivalent with $G$ being $d$-regular and $\mathcal{K}$ is just the set of edges of $G$.

Let $\epsilon \in [0,1/d]$ be a fixed constant. In this paper, we study the mixing rate of the following random walk $W_\epsilon$ on the vertices of $G$. Start with an arbitrary vertex. In the first step, the current vertex picks one of its $d(l-1)$ neighbors uniformly at random. In each subsequent step, the walk can stay in the same clique (from $\mathcal{K}$) as the most recent used edge with probability $\epsilon$ by uniformly choosing one of the $l-1$ neighbors in the current clique, or else it can leave the clique containing the most recent edge with probability $1-\epsilon$ by uniformly choosing one of its remaining $(d-1)(l-1)$ neighbors. After the first step, the probability of choosing a neighbor in the current clique is $p_s:=\frac{\epsilon}{l-1}$ (we call $p_s$ the {\em staying probability}) and the probability of choosing a neighbor in a different clique is $p_l:=\frac{1-\epsilon}{(d-1)(l-1)}$ (we call $p_l$ the {\em leaving probability}).

When $l=2$ and $\epsilon=1/d$, $W_{\epsilon}$ is the usual random walk on the vertices of a regular graph $G$ whose behavior is well studied \cite{AF, L}. In particular, it is known that the mixing rate (see Definition \ref{defmixrate}) of such random walk is $\rho=\frac{\max (|\lambda_2|, |\lambda_n|)}{d(l-1)}$ (see \cite[Corollary 5.2]{L}).

\begin{rk}\label{def:CliquewiseNBRW}
When $\epsilon=0$ and $l\ge 2$, $W_{\epsilon}$ is what we call a {\em cliquewise} non-backtracking random walk on $G$. This means that in each step, the walk cannot stay in the same clique it came from. i.e. $W_0^{(k)}$ is the set of $(w_0,w_1,\cdots,w_k)$ such that $w_t\in V$, $w_{t-1}w_t\in E$ for all $t\in [k]$, $w_{t-1}\neq w_{t+1}$; $w_{t-1}w_t$ and $w_tw_{t+1}$ cannot be both in the same clique from $\mathcal{K}$. Furthermore when $\epsilon=0$ and $l=2$, $W_{\epsilon}$ is a non-backtracking random walk on $G$ whose behavior has been studied extensively in recent years\cite{ABLS,AL,FH,KMMNSZZ}. In particular, Theorem 1.1 in \cite{ABLS} will be a special case of our Theorem \ref{main1}.
\end{rk}

Define the $k$-steps transition probability of $W_{\epsilon}$ as follows:
$$
\widetilde{P}_{uv}^{(k)}:=\mathbb{P}(X_k=v|X_0=u).
$$

\begin{definition}\label{defmixrate}
The mixing rate of this random walk with respect to the uniform distribution is defined by 
$$
\widetilde{\rho}(G):=\limsup_{k\rightarrow\infty}\max_{u,v\in V}\Big|\widetilde{P}_{uv}^{(k)}-\frac{1}{n}\Big|^{1/k}.
$$
\end{definition}
Note that if $\widetilde{\rho}(G)<1$, then the $k$-steps transition probability distribution converges to uniform distribution as $k\rightarrow\infty$ (i.e. the total variance of the probability distribution $\widetilde{P}_{uv}^{(k)}$ and uniform distribution decreases exponentially). 

Define $\psi:[0,\infty)\rightarrow\mathbb{R}$ by:
\begin{equation}\label{LimitFn}
\psi(x):=\begin{cases}  
1        &\mbox{if } 0\le x\le 1\\
x+\sqrt{x^2-1} &\mbox{if } x\ge 1
\end{cases}
\end{equation}

Let $d(l-1)=\lambda_1\ge \lambda_2\ge \lambda_3\ge \cdots\ge \lambda_n$ be the eigenvalues of the adjacency matrix of $G$. Let $N$ be the vertex-clique incidence matrix of $G$ corresponding to the clique partition $\mathcal{K}$. The rows of $N$ are indexed by the vertices of $G$ and the columns are indexed by the cliques in $\mathcal{K}$. For any $x\in V(G)$ and $K\in \mathcal{K}$, $N(x,K)=1$ if $x$ is contained in $K$ and $0$ otherwise. It is straightforward that the adjacency matrix of $G$ equals $NN^{t}-dI$. This implies that $\lambda_n\geq -d$. 

The following are the main results of our paper.

\begin{theorem}\label{main1}
Let $d\geq 2, l\geq 2$ be two integers and $\epsilon\in [0,1/d)$. Denote $\delta:=\frac{\epsilon(d-1)}{1-\epsilon}$. Let $G$ 
and $W_\epsilon$ be the $d(l-1)$-regular graph and the random walk defined above, respectively.

\begin{enumerate}

\item If $l(1-\delta)\le d$ with $d\ge 3$ and $l\ge 2$, then $W_\epsilon$ converges to the uniform distribution, and its mixing rate, $\tilde{\rho}$, satisfies:
\begin{eqnarray}\label{LimitFnless}
\tilde{\rho}=\sqrt{\frac{1-\delta}{(d-1+\delta)(l-1)}}\psi\Big(\frac{\lambda}{2\sqrt{(l-1)(1-\delta)(d-1+\delta)}}\Big)
\end{eqnarray}
where $\lambda:=\max_{i=2,n}|\lambda_i-(l-2)(1-\delta)|$. Moreover, if we treat $\tilde{\rho}$ as a function of $\epsilon$ on $[0,1/d]$, then $\tilde{\rho}(\epsilon)$ is continuous on $[0,1/d]$. 

\item If $l(1-\delta)>d$ with $d\ge 2$ and $l\ge 2$, then $W_\epsilon$ converges to the uniform distribution, and its mixing rate, $\tilde{\rho}$, satisfies: 
\begin{eqnarray}\label{LimitFnGreater}
\tilde{\rho}=\sqrt{\frac{1-\delta}{(d-1+\delta)(l-1)}}\psi\Big(\frac{\hat{\lambda}}{2\sqrt{(l-1)(1-\delta)(d-1+\delta)}}\Big)
\end{eqnarray}
where $\hat{\lambda}:=\max_{i:2\leq i\leq n;~\lambda_i\neq-d}|\lambda_i-(l-2)(1-\delta)|.$

\end{enumerate}
\end{theorem}

By taking $\epsilon=0$ in the previous theorem, we obtain the following results.
\begin{corollary}\label{maincor1}
Let $d\geq 2$ and $l\geq 2$ be two integers and let $G$ be a connected and non-bipartite $d(l-1)$-regular graph defined above. Let $\lambda:=\max_{i=2,n}|\lambda_i-(l-2)|$ and $\hat{\lambda}:=\max_{i:2\leq i\leq n;~\lambda_i\neq -d}|\lambda_i-(l-2)|$.
\begin{enumerate}

\item If $d\geq l$ with $d\ge 3$ and $l\ge 2$, then a cliquewise non-backtracking random walk as defined in Remark \ref{def:CliquewiseNBRW} on $G$ converges to the uniform distribution, and its mixing rate, $\widetilde{\rho}$, satisfies:
\begin{equation}\label{dlarger}
\widetilde{\rho}=\frac{1}{\sqrt{(d-1)(l-1)}}\psi\Big(\frac{\lambda}{2\sqrt{(d-1)(l-1)}}\Big)
\end{equation}

\item If $d<l$, with $d\ge 2$ and $l\ge 3$, then a cliquewise non-backtracking random walk as defined in Remark \ref{def:CliquewiseNBRW} on $G$ converges to the uniform distribution, and its mixing rate, $\tilde{\rho}$, satisfies:
\begin{equation}\label{dsmaller}
\widetilde{\rho}=\frac{1}{\sqrt{(d-1)(l-1)}}\psi\Big(\frac{\hat{\lambda}}{2\sqrt{(d-1)(l-1)}}\Big)
\end{equation}
\end{enumerate}
\end{corollary}

\begin{rk}
We mention the statement {\em $\tilde{\rho}(\epsilon)$ is continuous on $[0,1/d]$} in Theorem \ref{main1} because $\tilde{\rho}=\tilde{\rho}(\epsilon)$ appearing in \eqref{LimitFnless} is a continuous function of $\delta\in [0,1)$ and hence of $\epsilon\in[0,1/d)$. On the other hand, if $\epsilon=1/d$, then $W_{1/d}$ is the simple random walk with mixing rate 
$$
\tilde{\rho}(1/d)=\frac{\max_{i=2,n}|\lambda_i|}{d(l-1)}
$$
So the statement actually means that $\tilde{\rho}(\epsilon)$ is left continuous at $\epsilon=1/d$.
\end{rk}

\begin{rk}
A special case of Corollary \ref{maincor1} is $l=2$, which is exactly the non-backtracking random walk on a $d$-regular graph $G$ defined in \cite{ABLS}. Theorem 1.1 in \cite{ABLS} will be obtained from Corollary \ref{maincor1} by taking $l=2$. 
\end{rk}	

\begin{rk}
The random walk in Theorem \ref{main1} is not necessarily a non-backtracking random walk for $0<\epsilon\le 1/d$. Because for every step, $W_\epsilon$ can choose the same clique of its last step with probability $\epsilon$, and then choose its last position with probability $1/(l-1)$. This means that we permit a ``little" backtracking in each step with probability $\epsilon/(l-1)$. However if $\epsilon=0$, we have a cliquewise non-backtracking random walk defined in Remark \ref{def:CliquewiseNBRW}.
\end{rk}

\section{Proofs of main results}

\begin{proof}[Proof of Theorem \ref{main1}]

Let $G'$ be the bipartite vertex-clique incidence graph of $G$. More precisely, $G'$ will have color classes $V=V(G)$ and $\mathcal{K}$ with $x\in V$ and $K\in \mathcal{K}$ being adjacent if and only if $x\in K$. Clearly, $G'$ is a bipartite $(d,l)$-biregular graph with each vertex in $V$ having degree $d$ and each vertex in $\mathcal{K}$ having degree $l$. A $k$-step random walk $x_0\rightarrow x_1\rightarrow\cdots\rightarrow x_k$ on $G$ is equivalent to a $2k$-step random walk $x_0\rightarrow K_0\rightarrow x_1\rightarrow K_1\rightarrow\cdots\rightarrow x_{k-1}\rightarrow K_k\rightarrow x_k$ on $G'$ such that $K_0,\dots, K_k\in \mathcal{K}$, $x_ix_{i+1}\in K_i$ for $0\leq i\leq k-1$. Also, by our setting $K_{i}=K_{i+1}$ with probability $\epsilon$, for $0\leq i\leq k-1$; and $x_i\neq x_{i+1}$ for $0\leq i\leq k-1$. 

We define the weight of walk on $G'$ to be $\delta^m$ if this walk backtracks exactly $m$ times. If $m=0$, the weight of such a walk will be $1$  and the walk is non-backtracking. If $\epsilon=0$ or $1/d$, which implies $\delta=0$ or $1$ respectively, then the weight will reduce to the number of walks. A walk in $G'$ can only backtrack on the color class corresponding to the vertices of $G$. This means that consecutive steps of the form $K\rightarrow x \rightarrow K$ are allowed, but consecutive steps of the form $x\rightarrow K\rightarrow x$ are forbidden. If $x_0\rightarrow K_0\rightarrow x_1\rightarrow K_1\rightarrow\cdots\rightarrow x_k\rightarrow K_k$ is a walk of length $2k$ in $G'$ that backtracks exactly $m$ times, this means that in the corresponding walk on $G$ with $k$ steps: $x_0\rightarrow x_1\rightarrow\cdots\rightarrow x_k$, $m$ will be number of steps the walk stays in the same clique it came from. 

Let $A$ be the adjacency matrix of $G'$, where the first rows and columns are indexed after the vertices of $G$. If $N$ is the vertex-clique incidence matrix defined before Theorem \ref{main1}, then $A=\begin{bmatrix} 0 & N\\ N^t & 0\end{bmatrix}$. Define $A^{(k)}$ as the matrix whose rows and columns are indexed by the vertices of $G'$, where $A^{(k)}_{x,y}$ equals the sum of the weights of all walks of length $k$ from $x$ to $y$ in $G'$. Let  $R^{(k)}$ be the upper left $n\times n$ principal matrix of $A^{(2k)}$ and $V^{(k)}$ be the upper left $n\times n$ principal matrix of $A^{(2k-1)}A$.

We claim that the probability transition matrix corresponding to the random walk $W_{\epsilon}$ equals 
\begin{equation}\label{eq:RepresOfProbTransMx}
P^{(k)}=\frac{R^{(k)}}{d(l-1)\big((d-1+\delta)(l-1)\big)^{k-1}}.
\end{equation}

Recall that $W_{\epsilon}$ is a random walk on the vertices of $G$ with a clique partition $\mathcal{K}$ that works as follows. After picking a neighbor at random in the first step, the walk will continue by picking a neighbor of the current vertex with the probability of choosing a neighbor in the current clique being $p_s:=\frac{\epsilon}{l-1}$ and the probability of choosing a neighbor in a different clique being $p_l:=\frac{1-\epsilon}{(d-1)(l-1)}$. Also, note that $\delta=\frac{\epsilon(d-1)}{1-\epsilon}=\frac{p_s}{p_l}$.

To prove \eqref{eq:RepresOfProbTransMx}, we only need to observe that by our definition of $\delta$, note that each walk  with $k$ steps containing exactly $m$ times staying in its previous clique is assigned a probability of 
\begin{align}\label{ProbabilitymstayTraditional}
\frac{1}{d(l-1)}p_s^{m}p_{l}^{k-1-m}&=\frac{1}{d(l-1)}\left(\frac{\epsilon}{l-1}\right)^m\left(\frac{1-\epsilon}{(d-1)(l-1)}\right)^{k-1-m}\\
&=\frac{1}{d(l-1)}\cdot \left(\frac{\epsilon(d-1)}{1-\epsilon}\right)^{m} \cdot \frac{1}{\left((d-1)(l-1)+\frac{(d-1)(l-1)\epsilon}{1-\epsilon}\right)^{k-1}}\\
&=\frac{1}{d(l-1)}\cdot \frac{\delta^m}{\left((d-1+\delta)(l-1)\right)^{k-1}}.
\end{align}

Let $U_k(x)$ be the Chebyshev polynomials of the second kind of degree $k\ge 0$:
\begin{equation}
U_k(\cos(\theta))=\frac{\sin\big((k+1)\theta\big)}{\sin\theta}
\end{equation}
It is known that these polynomials satisfy the recurrence relation $U_{k+1}(x)=2xU_k(x)-U_{k-1}(x)$ and $U_{-1}(x)=0$, $U_0(x)=1$, $U_1(x)=2x$, $U_2(x)=4x^2-1$. 

For $k\geq 1$, define
\begin{equation}\label{eq:RelationChebyshev}
q_k(x):=\sqrt{(l-1)(1-\delta)(d-1+\delta)}U_k(x)+(l-2)(1-\delta)U_{k-1}(x)-\frac{(1-\delta)\sqrt{(l-1)(1-\delta)}}{\sqrt{d-1+\delta}}U_{k-2}(x)
\end{equation}
and 
\begin{equation}\label{eq:Repressmallq}
Q_k(x):=\sqrt{\big((l-1)(1-\delta)(d-1+\delta)\big)^{k-1}}q_k\Big(\frac{x-(l-2)(1-\delta)}{2\sqrt{(l-1)(1-\delta)(d-1+\delta)}}\Big).
\end{equation}

We claim that 
\begin{equation}\label{claimofUk}
R^{(k)}=Q_k(R),
\end{equation}
where $R=R^{(1)}$, i.e. the adjacency matrix of $G$.

To prove this claim, we will first obtain a recurrence relation for the matrices $A^{(k)}$ which will enable us to find a recurrence relation for the matrices $R^{(k)}$. The result will then follow by induction on $k$.

Clearly 
\begin{eqnarray}\label{eq:1stepA}
A^{(1)}=A. 
\end{eqnarray}

If $x\neq y\in V(G')$, we have $A^{(2)}_{x,y}=A^2_{x,y}$. If $x$ is a vertex of $G$, then $A^{(2)}_{x,x}=0=A^2_{x,x}-d$. If $K$ is a clique in $\mathcal{K}$ and $y\in K$ in $G$, then the weight of the walk: $K\rightarrow y\rightarrow K$ is $\delta$. Therefore the sum of the weights of all walks from $K$ to $K$ with length 2 is $l\delta$. Thus, $A^{(2)}_{K,K}=l\delta=A^2_{K,K}-l(1-\delta)$ implying that
\begin{eqnarray}\label{eq:2stepA}
A^{(2)}=A^2-\begin{pmatrix}
dI & 0\\
0 & l(1-\delta)I
\end{pmatrix}.
\end{eqnarray}

For $k\geq 2$, we claim that
\begin{eqnarray}\label{eq:k+1stepA}
A^{(k+1)}=A^{(k)}A-A^{(k-1)}\begin{pmatrix}
(d-1+\delta)I & 0\\
0 & (l-1)(1-\delta)I
\end{pmatrix}.
\end{eqnarray}

If $*\in V(G')=V\cup \mathcal{K}$ and $y\in V$, then $A^{(k)}A_{*,y}$ equals the sums of the weights of the walks $*=w_0, w_1, \dots, w_{k}=K$ of length $k$ from $*$ to $K$, where $K\in \mathcal{K}$ ranges through all neighbors of $y$ in $G'$. If $w_{k-1}\neq y$ then $*=w_0, w_1,\dots, w_{k}=K, w_{k+1}=y$ is a walk of length $k+1$ with no backtracking at the last step and the total weight of such walks is $A^{(k+1)}_{*,y}$. If $w_{k-1}=y$, then the weight of all walks $*=w_0, w_1,\dots, w_k$ with $w_{k-1}=y$ is $(d-1+\delta)A^{(k-1)}_{*,y}$, where $d-1$ comes from the $d-1$ choices of $w_{k}$ with $w_{k}\neq w_{k-2}$ and $\delta$ comes from the weight $\delta$ for those walks with $w_{k}=w_{k-2}$.  

If $*\in V(G')=V\cup \mathcal{K}$ and $K\in \mathcal{K}$, then $A^{(k)}A_{*,K}$ equals the sums of the weights of the walks $*=w_0, w_1, \dots, w_{k}=z$ of length $k$ from $*$ to $z$, where $z\in V$ ranges through all neighbors of $K$ in $G'$. If $w_{k-1}\neq K$ then $*=w_0, w_1,\dots, w_{k}=z, w_{k+1}=K$ is a walk of length $k+1$ with no backtracking at the last step and the total weight of such walks is $A^{(k+1)}_{*,K}$. If $w_{k-1}=K$, then the weight of all walks $*=w_0, w_1,\dots, w_k$ with $w_{k-1}=K$ is $(l-1)(1-\delta)A^{(k-1)}_{*,K}$, where $l-1$ comes from the $l-1$ choices of $w_k$ such that $w_k\neq w_{k-2}$ and $1-\delta$ comes from the fact that for every such walk, we need to keep the $\delta$-portion of backtracking at the last step. 

Using \eqref{eq:1stepA}, \eqref{eq:2stepA}, \eqref{eq:k+1stepA}, we obtain the following recurrence relations: 
$$
\begin{cases}
R^{(k+1)}=V^{(k+1)}-(d-1+\delta)R^{(k)}.\\
V^{(k+1)}=R^{(k)}(R^{(1)}+dI)-(l-1)(1-\delta)V^{(k)}.
\end{cases}
$$

With $R=R^{(1)}$, a simple calculation yields
\begin{equation}\label{eq:RecurOfU}
\begin{cases}
R^{(2)}=R^2-(l-2)(1-\delta)R-d(l-1)(1-\delta)I.\\
R^{(k+1)}=R^{(k)}R-(l-2)(1-\delta)R^{(k)}-(l-1)(1-\delta)(d-1+\delta)R^{(k-1)}.
\end{cases}
\end{equation}

We will use strong induction on $k$ to prove our claim \eqref{claimofUk}. For the base case $k=1$, by plugging in $U_{-1}(x)=0$, $U_0(x)=1$, $U_1(x)=2x$, $U_2(x)=4x^2-1$, we have
\begin{eqnarray*}
R^{(1)}&=&R=\sqrt{(l-1)(1-\delta)(d-1+\delta)}\frac{R-(l-2)(1-\delta)I}{\sqrt{(l-1)(1-\delta)(d-1+\delta)}}+(l-2)(1-\delta)I\\
&=&Q_1(R)
\end{eqnarray*}
Also, for $k=2$, we have that
\begin{eqnarray*}
R^{(2)}&=&R^2-(l-2)(1-\delta)R-d(l-1)(1-\delta)I\\
&=&\sqrt{(l-1)(1-\delta)(d-1+\delta)}\cdot\\
&&\Bigg(\sqrt{(l-1)(1-\delta)(d-1+\delta)}\bigg(4\Big(\frac{R-(l-2)(1-\delta)I}{2\sqrt{(l-1)(1-\delta)(d-1+\delta)}}\Big)^2-I\bigg)\\
&&+(l-2)(1-\delta)\frac{R-(l-2)(1-\delta)I}{\sqrt{(l-1)(1-\delta)(d-1+\delta)}}\\
&&-\frac{(1-\delta)\sqrt{(l-1)(1-\delta)}}{\sqrt{d-1+\delta}}I\Bigg)\\
&=&Q_2(R).
\end{eqnarray*}
For the induction step, assume that \eqref{claimofUk} is true for every $l\leq k$ and we will prove it for $k+1$. 
Because $q_k$ is a linear combination of Chebyshev polynomials, we have the recurrence relation
$q_{k+1}(y)=2yq_k(y)-q_{k-1}(y)$ for every $k\geq 1$. Therefore we have
\begin{align*}
R^{(k+1)}&=R^{(k)}R-(l-2)(1-\delta)R^{(k)}-(l-1)(1-\delta)(d-1+\delta)R^{(k-1)}\\
&=RQ_k(R)-(l-2)(1-\delta)Q_k(R)-(l-1)(1-\delta)(d-1+\delta)Q_{k-1}(R)\\
&=\sqrt{\big((l-1)(1-\delta)(d-1+\delta)\big)^{k}}\frac{R-(l-2)(1-\delta)I}{\sqrt{(l-1)(1-\delta)(d-1+\delta)}}\\
&\cdot q_k\Big(\frac{R-(l-2)(1-\delta)I}{2\sqrt{(l-1)(1-\delta)(d-1+\delta)}}\Big)\\
&-\sqrt{\big((l-1)(1-\delta)(d-1+\delta)\big)^{k}}q_{k-1}\Big(\frac{R-(l-2)(1-\delta)I}{2\sqrt{(l-1)(1-\delta)(d-1+\delta)}}\Big)\\
&=\sqrt{\big((l-1)(1-\delta)(d-1+\delta)\big)^{k}}q_{k+1}\Big(\frac{R-(l-2)(1-\delta)I}{2\sqrt{(l-1)(1-\delta)(d-1+\delta)}}\Big)\\
&=Q_{k+1}(R).
\end{align*}
This completes the induction process.

Let $1=\mu_1(k)$, $\mu_2(k)$, $\cdots$, $\mu_n(k)$ denote the eigenvalues of $P^{(k)}$, and let
$$
\mu(k):=\max\{|\mu_2(k)|,~\cdots,~|\mu_n(k)|\}.
$$
We claim that
\begin{eqnarray}
\frac{\mu(k)}{n}\le \max_{i,j}\Big|P^{(k)}_{i,j}-\frac{1}{n}\Big|\le \mu(k).
\end{eqnarray}
To see this, note that the unit vector $v_1:=\frac{1}{\sqrt{n}}(1,1,\cdots,1)$ is an eigenvector of $P^{(k)}$ corresponding to the eigenvalue $\mu_1=1$. Therefore
$$
\max_{i,j}\Big|P^{(k)}_{i,j}-\frac{1}{n}\Big|=\max_{i,j}\Big|\big<(P^{(k)}-v_1^Tv_1)e_i,e_j\big>\Big|\le \max_{|u|=|v|=1}\Big|\big<(P^{(k)}-v_1^Tv_1)u,v\big>\Big|=\mu(k).
$$
On the other hand
$$
\max_{i,j}\Big|P^{(k)}_{i,j}-\frac{1}{n}\Big|\ge \frac{1}{n}\sqrt{\sum_{i,j}\Big|P_{i,j}^{(k)}-\frac{1}{n}\Big|^2}=
\frac{1}{n}\sqrt{\tr[(P^{(k)}-v_1^Tv_1)^2]}=\frac{1}{n}\sqrt{\sum_{s=2}^{n}\mu_{s}^{2}(k)}\ge \frac{\mu(k)}{n}.
$$
Thus,
\begin{eqnarray}\label{eq:ReprOfMixingRateByEvalue}
\tilde{\rho}=\limsup_{k\rightarrow\infty}\mu(k)^{1/k}=\max_{2\le i\le n}\limsup_{k\rightarrow\infty}|\mu_i(k)|^{1/k}.
\end{eqnarray}
Using \eqref{eq:RepresOfProbTransMx} and that $R=R^{(1)}$ is the adjacency matrix of $G$, we get that
\begin{equation}\label{eq:ReprOfEvalue}
\mu_i(k)=\frac{1}{d(l-1)}\sqrt{\Big(\frac{1-\delta}{(d-1+\delta)(l-1)}\Big)^{k-1}}q_k\Big(\frac{\lambda_i-(l-2)(1-\delta)}{2\sqrt{(l-1)(1-\delta)(d-1+\delta)}}\Big)
\end{equation}
where $\lambda_i$ is the $i$-th largest eigenvalue of $R$. We will use the following lemma whose proof is contained in the Appendix. 

\begin{lemma}\label{mainlemma}
For $d>2$, $l\ge 2$ and $\delta\in [0,1)$, if $l(1-\delta)\le d$, then
\begin{eqnarray}\label{eq:MainLemmaLess}
\limsup_{k\rightarrow\infty}|q_k(y)|^{1/k}=\begin{cases}
1 & \mbox{if}~-1\le y\le 1\\
|y|+\sqrt{y^2-1} &\mbox{if}~|y|\ge 1.
\end{cases}
\end{eqnarray}
For $d\ge 2$, $l\ge 2$ and $\delta\in [0,1)$, if $l(1-\delta)>d$, then
\begin{equation}\label{eq:MainLemmaGreater}
\limsup_{k\rightarrow\infty}|q_k(y)|^{1/k}=\begin{cases}
1&\mbox{if} ~-1\le y\le 1\\
|y|+\sqrt{y^2-1} &\mbox{if} ~|y|>1, y\neq\frac{-d-(l-2)(1-\delta)}{2\sqrt{(l-1)(1-\delta)(d-1+\delta)}}\\
\sqrt{\frac{d-1+\delta}{(l-1)(1-\delta)}} &\mbox{if}~y=\frac{-d-(l-2)(1-\delta)}{2\sqrt{(l-1)(1-\delta)(d-1+\delta)}}
\end{cases}
\end{equation}
\end{lemma}

Thus, if $l(1-\delta)\le d$, by \eqref{eq:ReprOfMixingRateByEvalue}, \eqref{eq:ReprOfEvalue} and \eqref{eq:MainLemmaLess}, we obtain that
\begin{equation*}
\tilde{\rho}=\sqrt{\frac{1-\delta}{(d-1+\delta)(l-1)}}\psi\Big(\frac{\lambda}{2\sqrt{(l-1)(1-\delta)(d-1+\delta)}}\Big),
\end{equation*}
where $\lambda:=\max_{\lambda_2,\lambda_n}|\lambda_i-(l-2)(1-\delta)|$. This proves \eqref{LimitFnless}. 

The random walk $W_{\epsilon}$ will converge to uniform distribution if $\tilde{\rho}<1$. This would be implied by 
\begin{eqnarray}\label{eq:EquivMixRate}
\lambda<(d-1+\delta)(l-1)+(1-\delta)=d(l-1)-(l-2)(1-\delta).
\end{eqnarray}
  
We will verify \eqref{eq:EquivMixRate} in the following argument. If $\lambda_2\ge (l-2)(1-\delta)$, then since $G$ is connected, $\lambda_2<d(l-1)$ and therefore,
\begin{equation*}
\lambda_2-(l-2)(1-\delta)<d(l-1)-(l-1)(1-\delta)+(1-\delta)=(d-1+\delta)(l-1)+(1-\delta).
\end{equation*}
If $\lambda_2<(l-2)(1-\delta)$, then $|\lambda_2-(l-2)(1-\delta)|=(l-2)(1-\delta)-\lambda_2 \le (l-2)(1-\delta)-\lambda_n$. From the observation contained on page 2 just before Theorem \ref{main1}, we know that $\lambda_n\ge -d$. When $l>2$ and $d>2$, we get that
\begin{align*}
(l-2)(1-\delta)-\lambda_n&\le (l-2)(1-\delta)+d\\
&=(d-1+\delta)(l-1)+(1-\delta)-(d+2\delta-2)(l-2)\\
&<(d-1+\delta)(l-1)+(1-\delta).
\end{align*}
If $l=2$, then $d=d(l-1)$ and, since the graph is non-bipartite, $\lambda_n>-d$. Consequently, we have that
$$
(l-2)(1-\delta)-\lambda_n<(l-2)(1-\delta)+d\le (d-1+\delta)(l-1)+(1-\delta)
$$
which proves that $\tilde{\rho}<1$.

Lastly, if we treat $\tilde{\rho}$ as a function of $\epsilon$ on $[0,1/d]$, in order to verify that $\tilde{\rho}(\epsilon)$ is continuous on $[0,1/d]$, it suffices to verify that 
$$\lim_{\epsilon\rightarrow 1/d}\tilde{\rho}(\epsilon)=\max_{i:2\leq i\leq n}\frac{|\lambda_i|}{d(l-1)}
$$
In fact, by the fact that $\delta\rightarrow 1$ when $\epsilon\rightarrow 1/d$, we get that
\begin{eqnarray*}
\lim_{\epsilon\rightarrow 1/d}\tilde{\rho}(\epsilon)&=&\lim_{\delta\rightarrow 1}\sqrt{\frac{1-\delta}{(d-1+\delta)(l-1)}}\psi\Big(\frac{\lambda}{2\sqrt{(l-1)(1-\delta)(d-1+\delta)}}\Big)\\
&=&\lim_{\delta\rightarrow 1}\sqrt{\frac{1-\delta}{(d-1+\delta)(l-1)}}\frac{\lambda+\sqrt{\lambda^2-4(l-1)(1-\delta)(d-1+\delta)}}{2\sqrt{(l-1)(1-\delta)(d-1+\delta)}}\\
&=&\max_{i: 2\leq i\leq n}\frac{|\lambda_i|}{d(l-1)}.
\end{eqnarray*}
This finishes the proof of the case $l(1-\delta)\le d$.

If $l(1-\delta)>d$ and $\lambda_i=-d$ for some $i$, then equation \eqref{eq:MainLemmaGreater} of Lemma \ref{mainlemma} implies that
\begin{equation}
\limsup_{k\rightarrow\infty}|\mu_i(k)|^{1/k}=\limsup_{k\rightarrow\infty}|q_k(y_0)|^{1/k}=\sqrt{\frac{d-1+\delta}{(l-1)(1-\delta)}}<1,
\end{equation}
where 
$$
y_0=\frac{-d-(l-2)(1-\delta)}{2\sqrt{(l-1)(1-\delta)(d-1+\delta)}}.
$$
Because $\psi(y)\geq 1$ for all $y$, we deduce that since there exists some $i$ such that $\lambda_i>-d$ (as $G$ is not complete), then
$$
\tilde{\rho}=\sqrt{\frac{1-\delta}{(d-1+\delta)(l-1)}}\psi\Big(\frac{\hat{\lambda}}{2\sqrt{(l-1)(1-\delta)(d-1+\delta)}}\Big)
$$
where
$$
\hat{\lambda}:=\max_{i:2\leq i\leq n;~\lambda_i>-d}|\lambda_i-(l-2)(1-\delta)|.
$$

Note that the random walk $W_\epsilon$ converges to uniform distribution in this case as well. If $\lambda_i=\lambda_{i+1}=\cdots=\lambda_n=-d$ where $i$ is the smallest index with this property, then the mixing rate is
$$
\tilde{\rho}=\sqrt{\frac{1-\delta}{(d-1+\delta)(l-1)}}\psi\Big(\frac{\hat{\lambda}}{2\sqrt{(l-1)(1-\delta)(d-1+\delta)}}\Big)
$$
with $\hat{\lambda}=\max_{j=2,i-1}|\lambda_j-(l-2)(1-\delta)|$. 

To make sure that this is strictly less than 1, we need $\lambda<(d-1+\delta)(l-1)+(1-\delta)$. If $\lambda_2\ge (l-2)(1-\delta)$, then we have the same argument as in the proof of the case $l(1-\delta)\le d$ since $G$ is connected. If $\lambda_2<(l-2)(1-\delta) $, we have $|\lambda_2-(l-2)(1-\delta)|=(l-2)(1-\delta)-\lambda_2<(l-2)(1-\delta)+d$ as $G$ is non-complete. So we can use the same argument to show $(l-2)(1-\delta)+d\le (d-1+\delta)(l-1)+(1-\delta)$. So we have $|\lambda_2-(l-2)(1-\delta)|<(d-1+\delta)(l-1)+(1-\delta)$. On the other hand, for $\lambda_{i-1}$, we can use the same argument as that of $\lambda_2$ by noting that $-d<\lambda_{i-1}\le\lambda_2<d(l-1)$. 
\end{proof}

\begin{proof}[Proof of Corollary \ref{maincor1}]
Substitute $\delta$ by 0 in Theorem \ref{main1}.
\end{proof}

\section{Comparing the mixing rates of simple, non-backtracking and cliquewise non-backtracking random walks}

\subsection{The mixing rates of the usual random walk and the cliquewise non-backtracking random walk}


\begin{corollary}\label{cor:ComparisonBtwCliqueNBRWSimpleRW}
Under the conditions of Corollary \ref{maincor1}, if both of the following conditions 
\begin{enumerate}
\item $d\ge l$;
\item $\lambda:=\max_{i=2,n}|\lambda_i-(l-2)|\ge 2\sqrt{(d-1)(l-1)}$
\end{enumerate}
are true, then $\widetilde{\rho}<\rho$, where $\rho$ is the mixing rate of the simple random walk on $G$ and $\tilde{\rho}$ is the mixing rate of the cliquewise non-backtracking random walk on $G$.
\end{corollary}

\begin{proof}[Proof of Corollary  \ref{cor:ComparisonBtwCliqueNBRWSimpleRW}]
As mentioned in the introduction, it is known that $\rho=\frac{\lambda'}{d(l-1)}$ where $\lambda':=\max \{|\lambda_2|, |\lambda_n|\}$ (see \cite[Corollary 5.2]{L}). 

If $\max_{i=2,n}|\lambda_i-(l-2)|\ge 2\sqrt{(d-1)(l-1)}$, we have either $|\lambda_2-(l-2)|\ge (l-2)-\lambda_n$ with $|\lambda_2-(l-2)|\ge 2\sqrt{(d-1)(l-1)}$, or else $|\lambda_2-(l-2)|\le (l-2)-\lambda_n$ with $|(l-2)-\lambda_n\ge 2\sqrt{(d-1)(l-1)}$. We will discuss in three cases. 
\begin{itemize}
\item Case 1: $|\lambda_2-(l-2)|\ge 2\sqrt{(d-1)(l-1)}$ and $|\lambda_2-(l-2)|\ge (l-2)-\lambda_n$.
\item Case 2: $(l-2)-\lambda_n\ge 2\sqrt{(d-1)(l-1)}$ and $|\lambda_2-(l-2)|\le (l-2)-\lambda_n$ and $|\lambda_n|\ge |\lambda_2|$. 
\item Case 3: $(l-2)-\lambda_n\ge 2\sqrt{(d-1)(l-1)}$ and $|\lambda_2-(l-2)|\le (l-2)-\lambda_n$ and $|\lambda_n|\le  |\lambda_2|$.
\end{itemize}

For Case 1, we claim that this case implies $\lambda_2\ge l-2$. Because if otherwise then we have 
$$
(l-2)-\lambda_2=|\lambda_2-(l-2)|\ge (l-2)-\lambda_n
$$
therefore $\lambda_2\le \lambda_n\Rightarrow\lambda_2=\lambda_n$, which means that $G$ is complete, contradicts to our assumption on $G$. Therefore it is clear that $\lambda_2>|\lambda_n|$ as $l-2> 0$. So for this case, we have
\begin{eqnarray}
&&\widetilde{\rho}=\frac{\lambda_2-(l-2)+\sqrt{\lambda_2^2-
2(l-2)\lambda_2+l^2-4d(l-1)}}{2(d-1)(l-1)}\label{eq:Case1tilderho}\\
&&\rho=\frac{\lambda_2}{d(l-1)}\label{eq:Case1rho}
\end{eqnarray}
and
\begin{equation}\label{eq:lambda2gethanl}
\lambda_2\ge 2\sqrt{(d-1)(l-1)}+(l-2)\ge 2+l-2=l
\end{equation}
Note that, since the graph $G$ is connected, then
\begin{eqnarray}
&&\lambda_2<d(l-1)\\
&&\Rightarrow -4d(l-1)<-4\lambda_2\label{eq:-4lambda2}
\end{eqnarray}
Thus we have
\begin{eqnarray*}
\widetilde{\rho}&=&\frac{\lambda_2-(l-2)+\sqrt{\lambda_2^2-
2(l-2)\lambda_2+l^2-4d(l-1)}}{2(d-1)(l-1)}\\
&\le &\frac{\lambda_2-(l-2)+\sqrt{\lambda_2^2-
2(l-2)\lambda_2+l^2-4\lambda_2}}{2(d-1)(l-1)}~~~\mbox{by}~(\ref{eq:-4lambda2})\\
&=&\frac{\lambda_2-(l-2)+\sqrt{(\lambda_2-l)^2}}{2(d-1)(l-1)}\\
&=&\frac{\lambda_2-(l-2)+\lambda_2-l}{2(d-1)(l-1)}~~~\mbox{by}~(\ref{eq:lambda2gethanl})\\
&=&\frac{\lambda_2-(l-1)}{(d-1)(l-1)}=\frac{\lambda_2-(l-1)}{d(l-1)-(l-1)}\\
&\le &\frac{\lambda_2}{d(l-1)}~~~\mbox{by}~(\ref{eq:-4lambda2})\\
&=&\rho
\end{eqnarray*}
which provides Case 1. 

For Case 2, we have
\begin{eqnarray}
&&\widetilde{\rho}=\frac{|\lambda_n|+(l-2)+\sqrt{|\lambda_n|^2+
2(l-2)|\lambda_n|+l^2-4d(l-1)}}{2(d-1)(l-1)}\label{eq:Case2tilderho}\\
&&\rho=\frac{|\lambda_n|}{d(l-1)}\label{eq:Case2rho}
\end{eqnarray}
So it suffices to prove
\begin{equation}\label{eq:sufficeCdt}
\frac{|\lambda_n|+(l-2)+\sqrt{|\lambda_n|^2+
2(l-2)|\lambda_n|+l^2-4d(l-1)}}{2(d-1)}\le \frac{|\lambda_n|}{d}
\end{equation}
and, because of $d\ge l$ and the fact that $\lambda_n\ge -d$
\begin{equation}\label{eq:dgethanl}
d\ge |\lambda_n|\ge 2\sqrt{(d-1)(l-1)}-(l-2)\ge 2\sqrt{(l-1)(l-1)}-(l-2)=l
\end{equation}
Thus the left hand side of (\ref{eq:sufficeCdt}) is:
\begin{eqnarray*}
&&\frac{|\lambda_n|+(l-2)+\sqrt{|\lambda_n|^2+
2(l-2)|\lambda_n|+l^2-4d(l-1)}}{2(d-1)}\\
&\le &\frac{|\lambda_n|+(l-2)+\sqrt{|\lambda_n|^2+
2(l-2)|\lambda_n|+l^2-4|\lambda_n|(l-1)}}{2(d-1)}~~~\mbox{by}~(\ref{eq:dgethanl})\\
&=& \frac{|\lambda_n|+(l-2)+\sqrt{(|\lambda_n|-l)^2}}{2(d-1)}\\
&=& \frac{|\lambda_n|+(l-2)+|\lambda_n|-l}{2(d-1)}~~~\mbox{by (\ref{eq:dgethanl})}\\
&=& \frac{|\lambda_n|-1}{d-1}\le \frac{|\lambda_n|}{d}~~~\mbox{by}~(\ref{eq:dgethanl})
\end{eqnarray*}
which provides Case 2. 

For Case 3, we claim that this case still implies that $\lambda_2\ge l-2$ (by which we will derive the value of $\rho$). To see this, the fact $|\lambda_n|\le |\lambda_2|$ and the fact that $G$ is not complete imply that $\lambda_2\ge 0$. Further if $\lambda_2\le l-2$, we have that 
$$
l-2\ge \lambda_2\ge |\lambda_n|\ge2\sqrt{(d-1)(l-1)}-(l-2) 
$$
therefore $l-2\ge2\sqrt{(d-1)(l-1)}-(l-2)$, which implies $(l-2)^2\ge (d-1)(l-1)\ge (l-1)^2$ by the fact that $d\ge l$, which is impossible. So we have
\begin{eqnarray}
&&\widetilde{\rho}=\frac{|\lambda_n|+(l-2)+\sqrt{|\lambda_n|^2+
2(l-2)|\lambda_n|+l^2-4d(l-1)}}{2(d-1)(l-1)}\label{eq:Case3tilderho}\\
&&\rho=\frac{\lambda_2}{d(l-1)}\label{eq:Case3rho}
\end{eqnarray}
So it suffices to prove
\begin{equation}\label{eq:sufficesCdtCase3}
\frac{|\lambda_n|+(l-2)+\sqrt{|\lambda_n|^2+
2(l-2)|\lambda_n|+l^2-4d(l-1)}}{2(d-1)}\le \frac{\lambda_2}{d}
\end{equation}
In this case, we still have (\ref{eq:dgethanl}) holds. Therefore by the same argument of Case 2, the left hand side of (\ref{eq:sufficesCdtCase3}) is:
\begin{eqnarray*}
\frac{|\lambda_n|+(l-2)+\sqrt{|\lambda_n|^2+
2(l-2)|\lambda_n|+l^2-4d(l-1)}}{2(d-1)}\le  \frac{|\lambda_n|}{d}\le \frac{\lambda_2}{d}
\end{eqnarray*}
which ends the proof of Corollary \ref{cor:ComparisonBtwCliqueNBRWSimpleRW}. 
\end{proof}

\begin{rk}
Inequality (\ref{eq:dgethanl}) implies that if $|\lambda_n-(l-2)|\ge 2\sqrt{(d-1)(l-1)}$ for the case $d\ge l$, then $-d\le \lambda_n\le -l$. In particular, if $d=l$ then $\lambda_n=-d$.
\end{rk}

\begin{corollary}\label{cor:ComparisonBtwCNBRWandSRWlowerbd}
Based on the conditions of Corollary \ref{cor:ComparisonBtwCliqueNBRWSimpleRW} and the three cases provided in the proof of Corollary \ref{cor:ComparisonBtwCliqueNBRWSimpleRW}, we have the following lower bounds:
\begin{itemize}
\item For Case 1,
\begin{equation}\label{eq:Case1Lowerbd}
\frac{\widetilde{\rho}}{\rho}\ge \frac{d}{2(d-1)}-\frac{d(l-2)}{2(d-1)\big(2\sqrt{(d-1)(l-1)}+(l-2)\big)}
\end{equation}
\item For Case 2,
\begin{equation}\label{eq:Case2Lowerbd}
\frac{\widetilde{\rho}}{\rho}\ge \frac{d}{2(d-1)}+\frac{l-2}{2(d-1)}
\end{equation}
\item For Case 3,
\begin{equation}\label{eq:Case3Lowerbd}
\frac{\widetilde{\rho}}{\rho}\ge \frac{d}{2(d-1)}-\frac{d(l-2)}{2(d-1)\big(2\sqrt{(d-1)(l-1)}-(l-2)\big)}
\end{equation}
\end{itemize} 
\end{corollary}

\begin{proof}[Proof of Corollary \ref{cor:ComparisonBtwCNBRWandSRWlowerbd}]
For Case 1, we have, by (\ref{eq:Case1tilderho}) and (\ref{eq:Case1rho}) and the fact that $\psi(x)\ge x$, we have
\begin{eqnarray*}
&&\frac{\widetilde{\rho}}{\rho}\ge \frac{d\lambda_2-d(l-2)}{2(d-1)\lambda_2}=\frac{d}{2(d-1)}-\frac{d(l-2)}{2(d-1)\lambda_2}
\end{eqnarray*}
which attains its minimum at $\lambda_2=2\sqrt{(d-1)(l-1)}+(l-2)$, which provides (\ref{eq:Case1Lowerbd}). To see this bound is not trivial, just note that $2\sqrt{(d-1)(l-1)}+(l-2)>(l-2)$, therefore the right hand side of (\ref{eq:Case1Lowerbd}) is positive. 

For case 2, by (\ref{eq:Case2tilderho}) and (\ref{eq:Case2rho}) and the fact that $\psi(x)\ge x$, we have
\begin{eqnarray*}
&&\frac{\widetilde{\rho}}{\rho}\ge \frac{d|\lambda_n|+d(l-2)}{2(d-1)|\lambda_n|}=\frac{d}{2(d-1)}+\frac{d(l-2)}{2(d-1)|\lambda_n|}
\end{eqnarray*}
which attains its minimum at $|\lambda_n|=d$, which provides (\ref{eq:Case2Lowerbd}). 

For case 3, by (\ref{eq:Case3tilderho}) (\ref{eq:Case3rho}) and the facts that $\psi(x)\ge x$ and $|\lambda_n|\ge \lambda_2-2(l-2)$, we have
\begin{eqnarray*}
\frac{\widetilde{\rho}}{\rho}&\ge & \frac{d|\lambda_n|+d(l-2)}{2(d-1)\lambda_2}\ge  \frac{d\lambda_2-2d(l-2)+d(l-2)}{2(d-1)\lambda_2}\\
&=&\frac{d}{2(d-1)}-\frac{d(l-2)}{2(d-1)\lambda_2}
\end{eqnarray*}
since $\lambda_2\ge |\lambda_n|\ge 2\sqrt{(d-1)(l-1)}-(l-2)$. Therefore 
$$
\frac{\widetilde{\rho}}{\rho}\ge \frac{d}{2(d-1)}-\frac{d(l-2)}{2(d-1)\big(2\sqrt{(d-1)(l-1)}-(l-2)\big)}
$$
which provides (\ref{eq:Case3Lowerbd}). To see that this is not a trivial bound, it suffices to prove that
$$
2\sqrt{(d-1)(l-1)}-(l-2)>l-2
$$
This is clear because $d\ge l$, therefore
$$
2\sqrt{(d-1)(l-1)}-(l-2)\ge 2\sqrt{(l-1)(l-1)}-(l-2)=l>l-2
$$
which ends the proof of Corollary \ref{cor:ComparisonBtwCNBRWandSRWlowerbd}.
\end{proof}

\begin{rk}
In Corollary \ref{cor:ComparisonBtwCNBRWandSRWlowerbd}, if we set $l=2$, then all cases reduce to 
$$
\frac{\widetilde{\rho}}{\rho}\ge \frac{d}{2(d-1)}
$$
which is the same as Corollary 1.2 in \cite{ABLS}. 
\end{rk}

\subsection{Comparison of the mixing rates of the non-backtracking random walk and the cliquewise non-backtracking random walk}

\begin{corollary}\label{cor:ComparisonBtwCNBRWandNBRW}
 Based on the conditions of Corollary \ref{maincor1} with the case $d\ge l$, define the following five constants: 
\begin{eqnarray*}
A(d,l)&:=&\frac{2\big(d(l-1)-1\big)}{2(d-1)(l-1)+\sqrt{l-2}\sqrt{(d-1)(l-1)}\Big(\sqrt{l-2}+\sqrt{(l-6)+4\sqrt{(d-1)(l-1)}\Big)}}\\
B(d,l)&:=& \frac{1+\sqrt{1-\frac{4}{(d-1)(l-1)}}}{1+\sqrt{1-\frac{4}{d(l-1)-1}}}\\
C(d,l)&:=& \frac{\sqrt{d(l-1)-1}}{\sqrt{(d-1)(l-1)}}\\
D(d,l)&:=& \frac{2\big(d(l-1)-1\big)+\sqrt{l-2}\sqrt{d(l-1)-1}\Big(\sqrt{l-2}
+\sqrt{(l+2)+4\sqrt{d(l-1)-1}}\Big)}{2(d-1)(l-1)}\\
E(d,l)&:=& \frac{2\big(d(l-1)-1\big)}{\big(l-1\big)\big(d+\sqrt{d^2-4(d(l-1)-1)}\big)}\\
F(d,l)&:=& \frac{\big(d(l-1)-1\big)}{(d-1)(l-1)}\\
&&\cdot \frac{2\sqrt{d(l-1)-1}+(l-2)+\sqrt{l-2}\sqrt{l+2+
4\sqrt{d(l-1)-1}}}{2\sqrt{d(l-1)-1}+2(l-2)+2\sqrt{l-2}\sqrt{l-2+
2\sqrt{d(l-1)-1}}}
\end{eqnarray*}
Define $\rho'$ as the the mixing rate of a non-backtracking random walk on $G$, then if $\lambda:=\max_{\lambda_2,\lambda_n}|\lambda_i-(l-2)|\ge 2\sqrt{(d-1)(l-1)}$, then we have the following 5 cases: 
\begin{itemize}
\item Case 1: $|\lambda_2-(l-2)|\ge 2\sqrt{(d-1)(l-1)}$ and $|\lambda_2-(l-2)|\ge (l-2)-\lambda_n$, then the ratio:
\begin{eqnarray}\label{Case1Bd}
A(d,l)\le \frac{\widetilde{\rho}}{\rho'}\le B(d,l)
\end{eqnarray}
\item Case 2: $\lambda_n\in [-2\sqrt{d(l-1)-1}, l-2-2\sqrt{(d-1)(l-1)}]$; $\lambda_2\le 2\sqrt{d(l-1)-1}$, then the ratio: 
\begin{eqnarray}\label{Case2Bd}
C(d,l)\le \frac{\widetilde{\rho}}{\rho'}\le D(d,l) 
\end{eqnarray}
\item Case 3: $\lambda_n\in [-2\sqrt{d(l-1)-1}, l-2-2\sqrt{(d-1)(l-1)}]$; $\lambda_2\in[ 2\sqrt{d(l-1)-1},(l-2)+2\sqrt{(d-1)(l-1)}]$, then the ratio:
\begin{eqnarray}\label{Case3Bd}
A(d,l)\le \frac{\widetilde{\rho}}{\rho'}\le D(d,l) 
\end{eqnarray}
\item Case 4: $l\le d/4+1/d+1$; $\lambda_n\le -2\sqrt{d(l-1)-1}$; $|\lambda_n|\ge |\lambda_2|$, then the ratio:
\begin{eqnarray}\label{Case4Bd}
E(d,l)\le \frac{\widetilde{\rho}}{\rho'}\le D(d,l) 
\end{eqnarray}
\item Case 5: $l\le d/4+1/d+1$; $\lambda_n\le -2\sqrt{d(l-1)-1}$; $|\lambda_n|\le |\lambda_2|$ and $|\lambda_2-(l-2)|\le (l-2)-\lambda_n$ then the ratio:
\begin{eqnarray}\label{Case5Bd}
F(d,l)\le \frac{\widetilde{\rho}}{\rho'}\le D(d,l) 
\end{eqnarray}
\end{itemize}
Moreover, we have
\begin{eqnarray*}
A(d,l)\le B(d,l)\le 1 \le  C(d,l)\le D(d,l)
\end{eqnarray*} 
Furthermore if $d$ and $l$ satisfy the prerequisite of Case 4 and Case 5 (i.e. $l\le d/4+1/d+1$), we have
\begin{eqnarray*}
A(d,l)\le F(d,l)\le B(d,l)\le 1 \le  C(d,l)\le E(d,l)\le D(d,l)
\end{eqnarray*}
\end{corollary}

\begin{rk}\label{resonablecase45}
Because $\lambda_n\ge -d$, so for Case 4 and Case 5, in order to make that $\lambda_n\le -2\sqrt{d(l-1)-1}$ reasonable, we need the prerequisite that 
$$
l\le \frac{d}{4}+\frac{1}{d}+1
$$
to make $-d\le -2\sqrt{d(l-1)-1}$.
\end{rk}

Before proving Corollary \ref{cor:ComparisonBtwCNBRWandNBRW}, we need to prove that our classification in Corollary \ref{cor:ComparisonBtwCNBRWandNBRW} is reasonable.

\begin{claim}\label{Claim}
We have 
\begin{eqnarray*}
&&-2\sqrt{d(l-1)-1}\le (l-2)-2\sqrt{(d-1)(l-1)}\le l-2\\
&&< 2\sqrt{d(l-1)-1}\le (l-2)+2\sqrt{(d-1)(l-1)}
\end{eqnarray*}
which means that the five cases are reasonable.
\end{claim}

\begin{proof}[Proof of Claim \ref{Claim}]
The first inequality is by the fact that
$$
2\sqrt{d(l-1)-1}\ge 2\sqrt{(d-1)(l-1)}\ge 2\sqrt{(d-1)(l-1)}-(l-2)
$$
The third inequality is by the fact that
$$
2\sqrt{d(l-1)-1}\ge 2\sqrt{l(l-1)-1}\ge 2l> l-2
$$
To see the last inequality, we will prove that for every $x\in [0,l-2]$:
\begin{equation}\label{eq:claimineq}
2\sqrt{d(l-1)-1-x}+x\ge 2\sqrt{d(l-1)-1}
\end{equation}
Then setting $x=l-2$. In fact, we have the left hand side equals to the right hand side when $x=0$, then (\ref{eq:claimineq}) holds if the derivative of the left hand side is larger tan 0. This is clear since
$$
\frac{1}{\sqrt{d(l-1)-1-x}}\le \frac{1}{\sqrt{d(l-1)-1-(l-2)}}=\frac{1}{\sqrt{(d-1)(l-1)}}\le 1
$$
which provides Claim \ref{Claim}. 
\end{proof}

\begin{proof} [Proof of Corollary \ref{cor:ComparisonBtwCNBRWandNBRW}]
As showed in \cite{ABLS}, 
\begin{equation}\label{eq:1}
\rho'=\psi\Big(\frac{\lambda'}{2\sqrt{d(l-1)-1}}\Big)/\sqrt{d(l-1)-1}
\end{equation}
where $\psi$ is the same as our definition of $\psi$.

1. For case 1, we claim that case 1 implies that $\lambda_2\ge l-2$. Because if otherwise, then we have $(l-2)-\lambda_2\ge (l-2)-\lambda_n$, therefore $\lambda_2\le \lambda_n$, which means $\lambda_2=\lambda_n$, contradicts to the fact that $G$ is not complete. Therefore we have $\lambda_2-(l-2)\ge |\lambda_n|+(l-2)$, so $\lambda_2\ge |\lambda_n|$ and $\lambda_2\ge 2\sqrt{(d-1)(l-1)}+(l-2)\ge 2\sqrt{d(l-1)-1}$ by Claim \ref{Claim}. Thus we have
\begin{eqnarray}
&&\widetilde{\rho}=\frac{\lambda_2-(l-2)+\sqrt{\big(\lambda_2-
(l-2)\big)^2-4(d-1)(l-1)}}{2(d-1)(l-1)}\label{eq:case1tilderho}\\
&&\rho'=\frac{\lambda_2+\sqrt{\lambda_2^2-4\big(d(l-1)-1\big)}}{2\big(d(l-1)-1\big)}\label{eq:case1rho}
\end{eqnarray}
We claim that $\widetilde{\rho}/\rho'$ is increasing in $\lambda_2$. In fact, we will show that for $a:=l-2$, $b:=4(d-1)(l-1)$, $c:=4(d(l-1)-1)$:
\begin{equation}\label{eq:defoff}
f(x):=\frac{(x-a)+\sqrt{(x-a)^2-b}}{x+\sqrt{x^2-c}}
\end{equation}
is increasing on $x\ge 2\sqrt{(d-1)(l-1)}+(l-2)$. Consider the numerator of the derivative of $f(x)$:
\begin{eqnarray*}
&&\Big(1+\frac{x-a}{\sqrt{(x-a)^2-b}}\Big)\big(x+\sqrt{x^2-c}\big)\\
&&-\big((x-a)+\sqrt{(x-a)^2-b}\big)\Big(1+\frac{x}{\sqrt{x^2-c}}\Big)\\
&=&a\Big(1-\frac{\sqrt{\lambda^2-c}}{\sqrt{(\lambda-a)^2-b}}\Big)+\big(\sqrt{x^2-c}-\sqrt{(x-a)^2-b}\big)\\
&&+x(x-a)\Big(\frac{1}{\sqrt{(x-a)^2-b}}-\frac{1}{\sqrt{x^2-c}}\Big)+x\Big(\frac{\sqrt{x^2-c}}{\sqrt{(x-a)^2-b}}-\frac{\sqrt{(x-a)^2-b}}{\sqrt{x^2-c}}\Big)\\
&=&\big(\sqrt{x^2-c}-\sqrt{(x-a)^2-b}\big)\bigg(1-\frac{a}{\sqrt{(x-a)^2-b}}+\frac{x(x-a)}{\sqrt{(x-a)^2-b}\cdot \sqrt{x^2-c}}\\
&&+x\cdot \Big(\frac{\sqrt{x^2-c}+\sqrt{(x-a)^2-b}}{\sqrt{(x-a)^2-b}\cdot \sqrt{x^2-c}}\Big)\bigg)\\
&=&\frac{\sqrt{x^2-c}-\sqrt{(x-a)^2-b}}{\sqrt{x^2-c}\cdot \sqrt{(x-a)^2-b}}\Big(x(x-a)+(x-a)\sqrt{x^2-c}+x\sqrt{(x-a)^2-b}\Big)\\
&&+\big(\sqrt{x^2-c}-\sqrt{(x-a)^2-b}\big)
\end{eqnarray*}
Furthermore we claim that 
\begin{eqnarray}
&&x\ge a \label{eq:xgea}\\
&&\sqrt{x^2-c}-\sqrt{(x-a)^2-b}\ge 0\label{eq:sqrtge0}
\end{eqnarray}
therefore the last equation is non-negative, which provides our previous claim. In fact, (\ref{eq:xgea}) comes from the fact that $x\ge 2\sqrt{(d-1)(l-1)}+(l-2)\ge l-2=a$ by Claim \ref{Claim}. For (\ref{eq:sqrtge0}), note that
\begin{equation}\label{eq:rangeofxincre}
\sqrt{x^2-c} \ge \sqrt{(x-a)^2-b}\Leftrightarrow x^2-c\ge (x-a)^2-b\Leftrightarrow  x\ge \frac{a}{2}-\frac{b}{2a}+\frac{c}{2a}=\frac{l-2}{2}+2
\end{equation}
This is clear since $x\ge 2\sqrt{(d-1)(l-1)}+(l-2)>2+\frac{l-2}{2}$. Thus, $\widetilde{\rho}/\rho'$ attains its maximum at $\lambda_2=d(l-1)-1$ (since $G$ is not complete) and its minimum at $\lambda_2= 2\sqrt{(d-1)(l-1)}+(l-2)$, which provides (\ref{Case1Bd}). To see the upper bound $B(d,l)$ is less than 1, just note that the function $g(x):=1+\sqrt{1-4/x}$ is increasing in $x$ and the fact that $(d-1)(l-1)\le d(l-1)-1$. 

2. For case 2, we have that both $|\lambda_n|$ and $|\lambda_2|$ are no-larger than $2\sqrt{d(l-1)-1}$. And $\lambda_2-(l-2)\le 2\sqrt{(d-1)(l-1)}$ if $\lambda_2\ge l-2$ by Claim 3.7. If $\lambda_2\le l-2$, then $|\lambda_2-(l-2)|=(l-2)-\lambda_2\le (l-2)-\lambda_n=|\lambda_n-(l-2)|$. And $|\lambda_n-(l-2)|\ge 2\sqrt{(d-1)(l-1)}$. So we have the mixing rates:
\begin{eqnarray}
&&\widetilde{\rho}=\frac{|\lambda_n|+(l-2)+\sqrt{(|\lambda|_n+(l-2)|)^2-4(d-1)(l-1)}}{2(d-1)(l-1)}\label{eq:case2tilderho}\\
&&\rho'=\frac{1}{\sqrt{d(l-1)-1}}\label{eq:case2rho}
\end{eqnarray}
therefore $\widetilde{\rho}/\rho'$ attains its maximum at $|\lambda_n|=2\sqrt{d(l-1)-1}$ and its minimum at $|\lambda_n|=2\sqrt{(d-1)(l-1)}-(l-2)$, which provides (\ref{Case2Bd}). 

\begin{rk}
To get a simpler upper bound of $D(d,l)$, note that
\begin{eqnarray*}
D(d,l)&\le & \sqrt{d(l-1)-1}\frac{2|\lambda_n|+2(l-2)}{2(d-1)(l-1)}\Bigg|_{|\lambda_n|=2\sqrt{d(l-1)-1}}\\
&=& \frac{2\big(d(l-1)-1\big)+(l-2)\sqrt{d(l-1)-1}}{(d-1)(l-1)}
\end{eqnarray*}
where the inequality is by the fact that $\sqrt{\big(|\lambda_n|+(l-2)\big)^2-\cdots}\le |\lambda_n|+(l-2)$, and this bound is simpler than $D(d,l)$ but is not sharp. because if $l=2$, then the $D(d,2)$ is 1, but the this bound is 2. 
\end{rk}

3. For case 3, we have $|\lambda_n-(l-2)|\ge |\lambda_2-(l-2)|$ and $|\lambda_n|\le \lambda_2$. Thus we have
\begin{eqnarray}
&&\widetilde{\rho}=\frac{|\lambda_n|+(l-2)+\sqrt{\big(|\lambda_n|+(l-2)\big)^2-4(d-1)(l-1)}}{2(d-1)(l-1)}\label{eq:case3tilderho}\\
&&\rho'=\frac{\lambda_2+\sqrt{\lambda_2^2-4(d(l-1)-1)}}{2\big(d(l-1)-1\big)}\label{eq:case3rho}
\end{eqnarray}
therefore $\widetilde{\rho}/\rho'$ attains its maximum at $|\lambda_n|=2\sqrt{d(l-1)-1}$, $\lambda_2=2\sqrt{d(l-1)-1}$ and its minimum at $|\lambda_n|=2\sqrt{(d-1)(l-1)}-(l-2)$, $\lambda_2=l-2+2\sqrt{(d-1)(l-1)}$. Moreover, since (\ref{eq:case2tilderho}) and (\ref{eq:case2rho}) hold at $|\lambda_n|=2\sqrt{d(l-1)-1}$ and $\lambda_2=2\sqrt{d(l-1)-1}$, thus the corresponding discussion in Case 2 is valid, which yields the same upper bound with Case 2. For the lower bound, we have that at $|\lambda_n|=2\sqrt{(d-1)(l-1)}-(l-2)$, $\lambda_2=l-2+2\sqrt{(d-1)(l-1)}$,
\begin{eqnarray*}
&&\widetilde{\rho}=\frac{1}{\sqrt{(d-1)(l-1)}}
\end{eqnarray*}
which is equivalent to the lower bound of Case 1. Therefore Case 3 shares the same lower bound with Case 1.

4. For case 4, we firstly claim that $|\lambda_n-(l-2)|\ge |\lambda_2-(l-2)|$. To see this, if $\lambda_2\ge l-2$, we already have $|\lambda_n|\ge \lambda_2$, therefore
$$
|\lambda_n-(l-2)|=|\lambda_n|+(l-2)>\lambda_2-(l-2)=|\lambda_2-(l-2)|
$$
If $\lambda_2\le l-2$, we have
$$
|\lambda_n-(l-2)|=(l-2)-\lambda_n\ge (l-2)-\lambda_2=|\lambda_2-(l-2)|
$$
Moreover, we have $|\lambda_n-(l-2)|\ge 2\sqrt{(d-1)(l-1)}$ and $|\lambda_n|\ge 2\sqrt{d(l-1)-1}$, so we have the mixing rates: 
\begin{eqnarray}
&&\widetilde{\rho}=\frac{|\lambda_n|+(l-2)+\sqrt{\big(|\lambda_n|+(l-2)\big)^2-4(d-1)(l-1)}}{2(d-1)(l-1)}\label{eq:case4tilderho}\\
&&\rho'=\frac{|\lambda_n|+\sqrt{\lambda_n^2-4(d(l-1)-1)}}{2\big(d(l-1)-1\big)}\label{eq:case4rho}
\end{eqnarray}
We claim that $\widetilde{\rho}/\rho'$ is decreasing in $|\lambda_n|$. In fact, we will show that for $a:=l-2$, $b:=4(d-1)(l-1)$, $c:=4(d(l-1)-1)$, the function defined by:
\begin{equation}\label{eq:f(x)oflamedan}
f(x):=\frac{(x+a)+\sqrt{(x+a)^2-b}}{x+\sqrt{x^2-c}}
\end{equation}
is decreasing on $x\ge 2\sqrt{d(l-1)-1}$. Consider the numerator of the derivative of $f(x)$:
\begin{eqnarray*}
&&\Big(1+\frac{a}{\sqrt{(x+a)^2-b}}+\frac{x}{\sqrt{x^2-c}}+\frac{x}{\sqrt{(x+a)^2-b}}+\frac{x(x+a)}{\sqrt{(x+a)^2-b}\sqrt{x^2-c}}\Big)\\
&&\cdot \big(\sqrt{x^2-c}-\sqrt{(x+a)^2-b}\big)
\end{eqnarray*}
Furthermore we claim that $\sqrt{x^2-c}-\sqrt{(x+a)^2-b}\le 0$ therefore the numerator is non-positive, which provides our previous claim. To see this, note that
\begin{eqnarray*}
\sqrt{x^2-c}\le \sqrt{(x+a)^2-b}\Longleftrightarrow x\ge \frac{b-c-a^2}{2a}
\end{eqnarray*}
This is true since the last term is negative. Thus $\widetilde{\rho}/\rho'$ attains its maximum at $|\lambda_n|=2\sqrt{d(l-1)-1}$ and its minimum at $|\lambda_n|=d$, which provides (\ref{Case4Bd}). Moreover, (\ref{eq:case2tilderho}) and (\ref{eq:case2rho}) hold at $|\lambda_n|=2\sqrt{d(l-1)-1}$, this yields the same upper bound with Case 2. To see $E(d,l)\ge C(d,l)$, we treat $\widetilde{\rho}/\rho'$ where $\widetilde{\rho}$ defined by (\ref{eq:case4tilderho}) and $\rho'$ by (\ref{eq:case4rho}) as a function, $F$, of $|\lambda_n|$. Thus $F$ is decreasing on $[2\sqrt{d(l-1)-1},\infty)$ by (\ref{eq:f(x)oflamedan}). Moreover for this case we have that $d\ge 2\sqrt{d(l-1)-1}$ by the prerequisite and Remark \ref{resonablecase45}. Therefore we have
$$
E(d,l)=F(d)\ge \lim_{x\rightarrow\infty} F(x)=\frac{d(l-1)-1}{(d-1)(l-1)}\ge \frac{\sqrt{d(l-1)-1}}{\sqrt{(d-1)(l-1)}}=C(d,l)
$$
where the last inequality is by the fact that $d(l-1)-1\ge (d-1)(l-1)$.

5. For case 5, we firstly claim that Case 5 implies $\lambda_2\ge l-2$. In fact, $|\lambda_n|\le |\lambda_2|$ implies $\lambda_2\ge 0$, otherwise we have $\lambda_n\ge\lambda_2$, contradicts to the fact that $G$ is not complete. Then suppose $\lambda_2\le l-2$, we have that $\lambda_n\le -2\sqrt{d(l-1)-1}\Rightarrow|\lambda_n|\ge 2\sqrt{d(l-1)-1}$. On the other hand $|\lambda_n|\le \lambda_2\le l-2$, so we have $2\sqrt{d(l-1)-1}\le l-2$, which contradicts to Claim \ref{Claim}. So we have the mixing rates:
\begin{eqnarray}
&&\widetilde{\rho}=\frac{|\lambda_n|+(l-2)+\sqrt{\big(|\lambda_n|+(l-2)\big)^2-4(d-1)(l-1)}}{2(d-1)(l-1)}\\
&&\rho'=\frac{\lambda_2+\sqrt{\lambda_2^2-4(d(l-1)-1)}}{2\big(d(l-1)-1\big)}
\end{eqnarray}
Therefore $\widetilde{\rho}/\rho'\le \widetilde{\rho}(|\lambda_n|)/\rho'(|\lambda_n|)$, similar to Case 4, we have the same upper bound $D(d,l)$ with Case 4. For the lower bound, we have $\widetilde{\rho}/\rho'\ge \widetilde{\rho}(|\lambda_n|)/\rho'\big(|\lambda_n|+2(l-2)\big)$ by the fact that $\lambda_2-(l-2)\le |\lambda_n|+(l-2)$. We claim that $\widetilde{\rho}(|\lambda_n|)/\rho'\big(|\lambda_n|+2(l-2)\big)$ is increasing in $|\lambda_n|$. In fact, we define 
$$
g(x):=\frac{d(l-1)-1}{(d-1)(l-1)}f(x+2a)
$$ where $f(x)$ is defined in (\ref{eq:defoff}). Therefore 
\begin{eqnarray*}
&&g(|\lambda_n|)=\frac{d(l-1)-1}{(d-1)(l-1)} f(|\lambda_n|+2a)\\
&=&\frac{d(l-1)-1}{(d-1)(l-1)}\cdot \frac{|\lambda_n|+a+\sqrt{(|\lambda_n|+a)^2-b}}{|\lambda_n|+2a+\sqrt{(|\lambda_n|+2a)^2-c}}\\
&=&\frac{\frac{|\lambda_n|+(l-2)+\sqrt{\big(|\lambda_n|+(l-2)\big)^2-4(d-1)(l-1)}}{2(d-1)(l-1)}}{\frac{|\lambda_n|+2(l-1)+\sqrt{\big(|\lambda_n|+2(l-1)\big)^2-4\big(d(l-1)-1\big)}}{2\big(d(l-1)-1\big)}}\\
&=&\frac{\widetilde{\rho}(|\lambda_n|)}{\rho'\big(|\lambda_n|+2(l-2)\big)}
\end{eqnarray*}
We have that $f(x)$ is increasing on $[2\sqrt{(d-1)(l-1)}+(l-2),\infty)$, therefore $g(x)$ is increasing on $[2\sqrt{(d-1)(l-1)}-(l-2),\infty)$. And $|\lambda_n|\in [2\sqrt{d(l-1)-1},d]\subset[2\sqrt{(d-1)(l-1)}-(l-2),\infty)$ by Claim 11. Therefore it attains its minimum at $|\lambda_n|=2\sqrt{d(l-1)-1}$, which provides (\ref{Case5Bd}). To see that $A(d,l)\le F(d,l)$, note that 
\begin{eqnarray*}
&&A(d,l)=\frac{d(l-1)-1}{(d-1)(l-1)}f\big(2\sqrt{(d-1)(l-1)}+(l-2)\big)\\
&=& g\big(2\sqrt{(d-1)(l-1)}-(l-2)\big)\le g\big(2\sqrt{d(l-1)-1}\big)=F(d,l)
\end{eqnarray*}
On the other hand, in order to prove $F(d,l)\le B(d,l)$, it suffices to prove that $F(d,l)\le B(d,l)$ for $l\ge 3$, since if $l=2$, then $F(d,l)= B(d,l)=1$. We firstly assume that
$$
d\notin (\frac{l-2\sqrt{2l-3}}{l-1}+1,\frac{l+2\sqrt{2l-3}}{l-1}+1)
$$
which means $(d-1)(l-1)\notin(l-2\sqrt{2l-3},l+2\sqrt{2l-3})$, assume $x_1:=(d-1)(l-1)$; $x_2:=l-2$, we have
\begin{eqnarray*}
&&x_1\notin (x_2+2-2\sqrt{2x_2+1},x_2+2+2\sqrt{2x_2+1})\\
&&\Longleftrightarrow x_1^2-(2x_2+4)x_1+x_2^2-4x_2\ge 0\\
&&\Longleftrightarrow (x_1-x_2)^2\ge 4(x_1+x_2)\\
&&\Longleftrightarrow x_1-x_2\geq 2\sqrt{x_1+x_2}\\
&&\Longleftrightarrow d(l-1)-1-2(l-2)\ge 2\sqrt{d(l-1)-1}
\end{eqnarray*}
Therefore we have
\begin{eqnarray*}
&&B(d,l)=\frac{d(l-1)-1}{(d-1)(l-1)}f\big(d(l-1)-1\big)\\
&=&g\big(d(l-1)-1-2(l-2)\big)\ge g(2\sqrt{d(l-1)-1})=F(d,l)
\end{eqnarray*}
To see the last inequality, we have that $f(x)$ is actually increasing on 
$$
[\max\{l-2,~\frac{l-2}{2}+2\},~\infty)
$$ 
by (\ref{eq:xgea}) and (\ref{eq:rangeofxincre}) in the proof of Case 1. And it is easy to verify that $2\sqrt{d(l-1)-1}+2(l-2)\ge \max\{l-2,~\frac{l-2}{2}+2\}$, therefore $f(x)$ is increasing on $[2\sqrt{d(l-1)-1}+2(l-2),\infty)$, hence $g(x)$ is increasing on $[2\sqrt{d(l-1)-1},\infty)$. Moreover, since $(l+2\sqrt{2l-3})/(l-1)+1$ is decreasing for $l\ge 3$, to see this, substituting $l$ by $k+1$ with $k\ge 2$, we have
$$
\frac{l+2\sqrt{2l-3}}{l-1}+1=2+\frac{1}{k}+2\sqrt{\frac{2}{k}-\frac{1}{k^2}}
$$
so we have 
$$
\max_{l\ge 3}\frac{l+2\sqrt{2l-3}}{l-1}+1=\frac{5}{2}+\sqrt{3}<5
$$
which means that if $d\ge 5$, then $(d-1)(l-1)\notin[l-2\sqrt{2l-3},l+2\sqrt{2l-3}]$, repeat our discussion we get $A(d,l)\le F(d,l)\le B(d,l) $. If $d< 5$, we have only three cases: $d=4$, $l=4$ or $d=4$, $l=3$ or $d=3$, $l=3$. However none of these cases satisfies the condition $l\le d/4+1/d+1$, which proves the bound $A(d,l)\le F(d,l)\le B(d,l)$.
\end{proof}

\begin{rk}
If we set $l=2$ in Corollary \ref{cor:ComparisonBtwCNBRWandNBRW}, then all upper bounds and lower bounds reduce to 1.
\end{rk}

\begin{rk}
From Corollary \ref{cor:ComparisonBtwCNBRWandNBRW}, we have that the cliquewise non-backtracking random walk mixes faster than non-backtracking random walk in Case 1 and slower in Case 2 and Case 4. 
\end{rk}

\begin{corollary}\label{cor:Comparisonoftilderhorhorho'}
Based on the conditions of Corollary \ref{maincor1}, assume $\rho$ to be the mixing rate of a simple random walk on $G$, $\rho'$ to be the mixing rate of a non-backtracking random walk on $G$, then if $\lambda:=\max_{\lambda_2,\lambda_n}|\lambda_i-(l-2)|\le 2\sqrt{(d-1)(l-1)}$ and $d(l-1)=n^{o(1)}$ as $n\rightarrow\infty$, then
\begin{align}
&\frac{\widetilde{\rho}}{\rho}\le \frac{d(l-1)}{2(d-1)(l-1)(1-o(1))+(l-2)\sqrt{(d-1)(l-1)}}\le \frac{d}{2(d-1)}(1+o(1))\label{eq:tilderhooverrhoupp}\\
&\frac{\widetilde{\rho}}{\rho}\ge 
\frac{d(l-1)}{2(d-1)(l-1)+(l-2)\sqrt{(d-1)(l-1)}}\label{eq:tilderhooverrholower}\\
&A(d,l)\le \frac{\widetilde{\rho}}{\rho'}\le C(d,l)\label{eq:tilderhooverrho'}
\end{align}
where $A(d,l)$ and $C(d,l)$ are defined in Corollary \ref{cor:ComparisonBtwCNBRWandNBRW}. 
\end{corollary}

\begin{proof}[Proof of Corollary \ref{cor:Comparisonoftilderhorhorho'}]
Define $\lambda':=\max (|\lambda_2|, |\lambda_n|)$, then for $\lambda\le 2\sqrt{(d-1)(l-1)}$, we have
\begin{eqnarray}
&&\widetilde{\rho}=\frac{1}{\sqrt{(d-1)(l-1)}}\label{tilderho}\\
&&\rho=\frac{\lambda'}{d(l-1)}\label{rho}\\
&&\rho'=\psi\Big(\frac{\lambda'}{2\sqrt{d(l-1)-1}}\Big)/\sqrt{d(l-1)-1}\label{rho'}
\end{eqnarray}
We need to use Theorem 1 in \cite{FL}: If the diameter of $G$ is $\ge 2k+2\ge 4$, then 
$$
\lambda_2>l-2+2\sqrt{(d-1)(l-1)}-\frac{2\sqrt{(d-1)(l-1)}-1}{k}
$$
Since the diameter of a $d(l-1)$-regular graph is at least $(1-o(1))\log_{d(l-1)-1}n$ (see the proof of Corollary 1.2 in [1]), we have that if $d(l-1)=n^{o(1)}$, then the diameter will converge to infinity, which is greater than 4 as $n$ large enough. So by this theorem, we have
\begin{eqnarray*}
\lambda'\ge \lambda_2\ge  (l-2)+2\sqrt{(d-1)(l-1)}(1-o(1))
\end{eqnarray*}
Thus substituting $\tilde{\rho}=1/\sqrt{(d-1)(l-1)}$ and $\rho=\lambda'/d(l-1)$ provides (\ref{eq:tilderhooverrhoupp}). 

On the other hand, we have
\begin{equation}\label{eq:lambda'le}
\lambda'\le (l-2)+2\sqrt{(d-1)(l-1)}
\end{equation}
which provides (\ref{eq:tilderhooverrholower}). The result (\ref{eq:tilderhooverrho'}) is by (\ref{eq:lambda'le}) and the fact that $\rho'\ge 1/\sqrt{d(l-1)-1}$. 
\end{proof}

\begin{rk}
A special case is $l=2$, then $\lambda=\lambda'$ and (\ref{eq:tilderhooverrhoupp}), (\ref{eq:tilderhooverrholower}) yield:
$$
\frac{d}{2(d-1)}\le \frac{\widetilde{\rho}}{\rho}\le \frac{d}{2(d-1)}\big(1+o(1)\big)
$$
which means $\widetilde{\rho}/\rho=\frac{d}{2(d-1)}+o(1)$, this is the same as the conclusion of Corollary 1.2 in \cite{ABLS}. And (\ref{eq:tilderhooverrho'}) yields $\widetilde{\rho}/\rho=1$.
\end{rk}

\section{Examples}

The readers less familiar with some of the notions used in this section (partial geometry, point graph, Latin square graphs) may wish to consult \cite{BH,GR,vLW}.

\begin{corollary}\label{cor:PG}
Suppose $G$ is the point graph of a partial geometry $pg(K,R,T)$, then if $R\ge 3$, $K\ge 3$, $R\ge K$, the cliquewise non-backtracking random walk converges slower than the non-backtracking random walk, in other words, $\tilde{\rho}> \rho'$. Moreover, 
$$
\tilde{\rho}=\frac{1}{K-1}
$$
\end{corollary}

\begin{proof}[Proof of Corollary \ref{cor:PG}]
 We have that $G$ is a $d(l-1)$-regular graph as defined in the first section, with $d=R$, $l=K$. We have
\begin{eqnarray}
\lambda_2 &=& K-1-T\\
\lambda_n &=& -R
\end{eqnarray}
Then we have $|\lambda_2-(K-2)|=T-1$, $|\lambda_n-(K-2)|=(R-1)+(K-1)$. Since $T\le \min\{K,R\}$, we have $|\lambda_2-(K-2)|\le |\lambda_n-(K-2)|$. Then by Corollary \ref{maincor1},
\begin{eqnarray}
\tilde{\rho}= \frac{1}{K-1}
\end{eqnarray}
On the other hand, since $|\lambda_n|=R\ge |K-1-T|=|\lambda_2|$, we have
\begin{eqnarray}
\rho' = \frac{1}{\sqrt{R(K-1)-1}}~\mbox{or}~\frac{R+\sqrt{R^2-4\big(R(K-1)-1\big)}}{2\big(R(K-1)-1\big)}
\end{eqnarray}
If $\rho'=\frac{1}{\sqrt{R(K-1)-1}}$, we have
\begin{eqnarray*}
\frac{\tilde{\rho}}{\rho'}&=&\frac{\sqrt{R(K-1)-1}}{K-1}\ge \frac{\sqrt{K(K-1)-1}}{K-1}\\
&\ge &\frac{\sqrt{(K-1)^2+(K-2)}}{K-1}> 1
\end{eqnarray*}
Further if $\rho'=\frac{R+\sqrt{R^2-4\big(R(K-1)-1\big)}}{2\big(R(K-1)-1\big)}$, we have
\begin{eqnarray*}
\frac{\tilde{\rho}}{\rho'}=\frac{2\big(R(K-1)-1\big)}{R(K-1)+(K-1)\sqrt{R^2-4\big(R(K-1)-1\big)}}
\end{eqnarray*}
To see $\tilde{\rho}/{\rho'}> 1$, note that this is equivalent to 
$$
\sqrt{R^2-4\big(R(K-1)-1\big)}< R-\frac{2}{K-1}
$$
which is equivalent to 
$$
\Big(1-\frac{1}{(K-1)^2}\Big)\big(R(K-1)-1\big)> 0
$$
which is clearly true. 
\end{proof}

\begin{rk}
In Corollary \ref{cor:PG}, the reason why we assume $K\ge 3$ ($l\ge 3$) instead of $K\ge 2$ ($l\ge 2$) is that if $K=2$, then $\lambda_n=-R=-d=-d(l-1)$, which means that this is a bipartite graph, which is not the case in Corollary \ref{maincor1}.
\end{rk}

\begin{corollary}\label{cor:PGRLessthanK}
Suppose $G$ is the point graph of a partial geometry $pg(K,R,T)$ with $K>2$, $R\ge 2$ and $R<K$, $T\ge 1$, then we have

1. The mixing rate, $\tilde{\rho}$, of a cliquewise non-backtracking random walk on $G$ satisfies
$$
\tilde{\rho}=\frac{1}{\sqrt{(K-1)(R-1)}}
$$

2. If $\frac{K-3}{4}+\frac{5}{4(K-1)}\le R$, then $\tilde{\rho}> \rho'$, i.e. cliquewise non-backtracking random walk mixes slower.

3. A special kind of partial geometry $pg(K,R,T)$ with $T=1$ is called a generalized quadrangle, denoted by $GQ(K-1,R-1)$. Suppose $G$ is the point graph of $GQ(K-1,R-1)$, then if $R\le K/4-1$, then $\tilde{\rho}< \rho'$, i.e. cliquewise non-backtracking random walk mixes faster.
\end{corollary}

\begin{proof}[Proof of Corollary \ref{cor:PGRLessthanK}]
1. We have, since $T\le R\le K-1$, 
\begin{eqnarray*}
\lambda_2&=&K-1-T\ge 0\\
\lambda_n&=&-R=-d
\end{eqnarray*}
Hence $|\lambda_2-(K-2)|=T-1\le 2\sqrt{(R-1)(K-1)}$ since $T\le \min\{K,R\}$. Thus by Corollary \ref{maincor1} and the fact that $pg(K,R,T)$ has only three eigenvalues, the mixing rate, $\tilde{\rho}$, of the cliquewise non-backtracking random walk satisfies
\begin{eqnarray}
\tilde{\rho}=\frac{1}{\sqrt{(K-1)(R-1)}}
\end{eqnarray}
which ends the proof of 1. 

2. If $\frac{K-3}{4}+\frac{5}{4(K-1)}\le R$, then we claim that $\max\{\lambda_2,|\lambda_n|\}\le 2\sqrt{R(K-1)-1}$. To see this, we surely have $|\lambda_n|=R\le 2\sqrt{(R-1)(K-1)}\le 2\sqrt{R(K-1)-1}$. For $\lambda_2$, we have $\lambda_2=K-1-T\le K-2$, therefore $K-2\le 2\sqrt{R(K-1)-1}\Rightarrow\lambda_2\le 2\sqrt{R(K-1)-1}$. And the condition 
$$
K-2\le 2\sqrt{R(K-1)-1}\Longleftrightarrow\frac{K-3}{4}+\frac{5}{4(K-1)}\le R
$$
Hence we have
\begin{eqnarray}\label{eq:rho'forKRT}
\rho'=\frac{1}{\sqrt{R(K-1)-1}}
\end{eqnarray}
and 
$$
\frac{\tilde{\rho}}{\rho'}=\frac{\sqrt{R(K-1)-1}}{\sqrt{(K-1)(R-1)}}=\frac{\sqrt{(K-1)(R-1)+(K-2)}}{\sqrt{(K-1)(R-1)}}>1
$$
which ends the proof of 2. 

3. We have $T=1$, hence
\begin{eqnarray*}
\lambda_2&=&K-2\\
\lambda_n&=&-R
\end{eqnarray*}
We firstly consider the case when $R\le (K-3)/4+5/\big(4(K-1)\big)$, we have
\begin{eqnarray}\label{lambda2ge2sqrt}
R\le \frac{K-3}{4}+\frac{5}{4(K-1)}\Longleftrightarrow \lambda_2=K-2\ge  2\sqrt{R(K-1)-1}
\end{eqnarray}
and
\begin{eqnarray}\label{lambdanlelambda2}
R\le \frac{K-3}{4}+\frac{5}{4(K-1)}\Longrightarrow R\le K-2~~(\mbox{i.e.}~|\lambda_n|\le \lambda_2)
\end{eqnarray}
then we have
\begin{eqnarray}\label{rho'forKRTwithless}
\rho'=\frac{K-2+\sqrt{ (K-2)^2-4\big( R(K-1)-1\big)}}{2\big(R(K-1)-1\big)}
\end{eqnarray}
Therefore 
\begin{equation}\label{tilderhooverrho'forKRTwithless}
\frac{\tilde{\rho}}{\rho'}=\frac{2\big(R(K-1)-1\big)}{\sqrt{(K-1)(R-1)}}\cdot \frac{1}{K-2+\sqrt{ (K-2)^2-4\big( R(K-1)-1\big)}}
\end{equation}
It can be shown that $\tilde{\rho}/\rho'$ is increasing with respect to $R$, in fact, from the last equation, define
$$
A(R):=\frac{2\big(R(K-1)-1\big)}{\sqrt{(K-1)(R-1)}}=2\sqrt{(K-1)(R-1)}+\frac{2(K-2)}{\sqrt{(K-1)(R-1)}}
$$
is increasing with respect to $R$ by noting that the function $f(x):=x+(K-2)/x$ is increasing for $x\ge \sqrt{K-2}$ and $\sqrt{(K-1)(R-1)}\ge \sqrt{K-2}$. On the other hand, 
\begin{eqnarray}\label{B(R)}
B(R):=\frac{1}{K-2+\sqrt{ (K-2)^2-4\big( R(K-1)-1\big)}}
\end{eqnarray}
is increasing with respect to $R$. Therefore $\tilde{\rho}/\rho'$ attains its maximum at $R= \frac{K-3}{4}+\frac{5}{4(K-1)}$, where (\ref{eq:rho'forKRT}) still holds. Hence we still have $\tilde{\rho}/\rho'\ge 1$ for $R= \frac{K-3}{4}+\frac{5}{4(K-1)}$ (i.e. upper bound). On the other hand, $\tilde{\rho}/\rho'$ attains its minimum at $R=2$, thus the lower bound:
\begin{eqnarray*}
&&\frac{\tilde{\rho}}{\rho'}=\frac{4K-6}{\sqrt{K-1}}\cdot \frac{1}{K-2+\sqrt{K^2-12K+16}}\\
&\le &\frac{4(K-1)}{\sqrt{K-1}}\cdot \frac{1}{K-1}=\sqrt{\frac{16}{K-1}}
\end{eqnarray*}
for $K\in \mathbb{N}$ such that $K^2-12K+16\ge 1$. Thus if $K\ge 18$, the lower bound is $<1$.

Further if $R\le K/4-1$, then we clearly have $R\le (K-3)/4+5/\big(4(K-1)\big)$, hence (\ref{lambda2ge2sqrt}), (\ref{lambdanlelambda2}), (\ref{rho'forKRTwithless}) and (\ref{tilderhooverrho'forKRTwithless}) hold. Since (\ref{tilderhooverrho'forKRTwithless}) is increasing with respect to $R$, so we only need to verify that (\ref{tilderhooverrho'forKRTwithless})$<1$ for $R=K/4-1$. We have, at $R=K/4-1$, 
\begin{eqnarray*}
(\ref{tilderhooverrho'forKRTwithless})=\frac{K^2-5K}{\sqrt{K^2-9K+8}}B\Big(\frac{K-4}{4}\Big)
\end{eqnarray*} 
where $B(R)$ is defined by (\ref{B(R)}). Since $B(R)$ is increasing with respect to $R$, therefore we have
\begin{eqnarray*}
B\Big(\frac{K-4}{4}\Big)<B\Big(\frac{K-3}{4}\Big)= \frac{1}{K-2+\sqrt{5}}<\frac{1}{K}
\end{eqnarray*}
Therefore we have
\begin{eqnarray}\label{73less}
(\ref{tilderhooverrho'forKRTwithless})<\frac{K^2-5K}{\sqrt{K^2-9K+8}}\cdot\frac{1}{K}
\end{eqnarray} 
We claim that (\ref{73less})$\le 1$ for $K\ge 17$ and hence (\ref{tilderhooverrho'forKRTwithless})$<1$ if $K\ge  17$, in fact, it suffices to prove that for $K\ge 17$,
$$
(K^2-5K)^2\le K^4-9K^3+8K^2
$$
which is obviously true. On the other hand, if $K<17$, then by our assumption $2\le R\le K/4-1\Rightarrow K\ge 12$. Thus we have six possibilities, they are: 
$$
(R,K)=\{(2,12),~(2,13),~(2,14),~(2,15),~(2,16),~(3,16)\}
$$
The corresponding $\tilde{\rho}/\rho'$ is:
$$
\frac{\tilde{\rho}}{\rho'}\approx\{0.904534,~0.810432,~0.744234,~0.69351,
~0.652692,~0.869771\}
$$
They are all less than 1, which ends the proof of Corollary \ref{cor:PGRLessthanK}. 
\end{proof}

\begin{corollary}\label{cor:GQqqpower}
Consider three special kinds of generalized quadrangle $GQ(q,1)$, $GQ(q,q)$, $GQ(q,q^2)$ where $q$ is a prime power, we have:

1. For $GQ(q,1)$, $\tilde{\rho}< \rho'$ for $q\ge 11$, and $\tilde{\rho}>  \rho'$ for $q\le 10$. 

2. For $GQ(q,q)$ and $GQ(q,q^2)$, then we always have $\tilde{\rho}>  \rho'$. 
\end{corollary}

\begin{proof}[Proof of Corollary \ref{cor:GQqqpower}]
1. For $GQ(q,1)$, we have $K=q+1$, $R=2$, then by the third part of Corollary \ref{cor:PGRLessthanK}, we have $\tilde{\rho}< \rho'$ for $K\ge 12$, therefore the same result holds for $q\ge 11$. If $q\le 10\Rightarrow K\le 11$, further if $K=11$, then by (\ref{tilderhooverrho'forKRTwithless}), we have $\tilde{\rho}/\rho'\approx 1.06947>1$. If else $K\le 10\Rightarrow \frac{K-3}{4}+\frac{5}{4(K-1)}\le 2=R\le K-1$, then by the second part of Corollary \ref{cor:PGRLessthanK}, we have $\tilde{\rho}>\rho'$. 

2. For $GQ(q,q)$ or $GQ(q,q^2)$, we have either $K=q+1$ and $R=q+1$, or else $K=q+1$ and $R=q^2+1$, both of which satisfy  $R\ge K$ and $R\ge 3$ and $K\ge 3$, then by Corollary \ref{cor:PG}, we have $\tilde{\rho}>\rho'$.
\end{proof}

\begin{corollary}\label{cor:Latinsquare}
Consider the partial geometry induced by a Latin square of size $l\times l$ with $l> 3$, i.e. two positions in the Latin square are said to be contained in one line if they are either in the same row or same column in the Latin square, or else they share the same value. So the corresponding partial geometry is $pg(K,R,T)$ with $K=l$, $R=3$, $T=2$. Therefore the point graph $G$ of this partial geometry is a $3(l-1)$-regular graph as defined in the first section. We have that $\tilde{\rho}>\rho'$ for $l\le 16$ (i.e. the cliquewise non-backtracking random walk mixes slower);  $\tilde{\rho}<\rho'$ for $l\ge 17$ (i.e. the cliquewise non-backtracking random walk mixes faster). Moreover, as $l\rightarrow\infty$, then $\tilde{\rho}/\rho'\sim\frac{3}{\sqrt{2l}}$. 
\end{corollary}

\begin{proof}[Proof of Corollary \ref{cor:Latinsquare}]
We have 
\begin{eqnarray}
\lambda_2 &=& l-3\\
\lambda_n &=& -3=-d
\end{eqnarray}
Since we have $d<l$ and $\lambda_n=-d$, then by Corollary \ref{maincor1}, we have 
$$
\tilde{\rho}=\frac{1}{\sqrt{2(l-1)}}\psi\Big(\frac{|\lambda_2-(l-2)|}{2\sqrt{2(l-1)}}\Big)
$$
and $|\lambda_2-(l-2)|=1<2\sqrt{2(l-1)}$, therefore 
\begin{eqnarray}
\tilde{\rho}=\frac{1}{\sqrt{2l-2}}
\end{eqnarray}
For $\rho'$, we clearly have 
$$
|\lambda_n|=3\le 2\sqrt{5}\le 2\sqrt{3(l-1)-1}=2\sqrt{d(l-1)-1}
$$
Then if $l\le 9+2\sqrt{14}\Leftrightarrow \lambda_2=l-3\le 2\sqrt{3(l-1)-1}=2\sqrt{d(l-1)-1}$, we have
\begin{eqnarray}
\rho'=\frac{1}{\sqrt{3l-4}}
\end{eqnarray}
Therefore $\tilde{\rho}>\rho'$. Further if $l>9+2\sqrt{14}$, we have
\begin{eqnarray}
\rho'=\frac{l-3+\sqrt{l^2-18l+25}}{6l-8}
\end{eqnarray}
therefore
\begin{eqnarray}\label{eq:tilderhooverrho'Latinsq}
\frac{\tilde{\rho}}{\rho'}=\frac{6l-8}{\sqrt{2l-2}}\cdot \frac{1}{l-3+\sqrt{l^2-18l+25}}
\end{eqnarray}
We have the following inequality for $\tilde{\rho}/\rho'$:
\begin{eqnarray*}
\frac{\tilde{\rho}}{\rho'}&=&\frac{6l-8}{\sqrt{2l-2}}\cdot \frac{1}{l-3+\sqrt{l^2-18l+25}}\\
&\le & \frac{6l-8}{\sqrt{2l-2}}\cdot \frac{1}{l-3}\le\frac{6l-8}{l-3}\cdot \frac{1} {\sqrt{2l-8/3}}\\
&=:&f(l)
\end{eqnarray*}
We claim that $f(l)<1$ for $l\ge 23$, therefore $\tilde{\rho}/\rho'<1$  for $l\ge 23$. In fact, we have
\begin{eqnarray*}
f(l)&=&\frac{6l-8}{\sqrt{2l-8/3}}\cdot \frac{1}{l-3}=\frac{3\sqrt{2}\sqrt{l-4/3}}{l-3} 
\end{eqnarray*}
We have
$$
f(l)<1\Longleftrightarrow 18(l-4/3)<(l-3)^2\Longleftarrow l>12+\sqrt{111} \Longleftarrow l\ge 23
$$
On the other hand, if $9+2\sqrt{14}<l<23$, we have six choices: $l=17$, 18, 19, 20, 21, 22. Then by (\ref{eq:tilderhooverrho'Latinsq}), we have $\tilde{\rho}/\rho'\approx 0.987437$, 0.857493, 0.780563, 0.724947, 0.681405, 0.645748 respectively. Moreover, as $l\rightarrow\infty$, $\tilde{\rho}/\rho'\sim\frac{3}{\sqrt{2l}}$ by (\ref{eq:tilderhooverrho'Latinsq}), which ends the proof of Corollary \ref{cor:Latinsquare}.
\end{proof}

Suppose that we have $t$ orthogonal $l\times l$ Latin squares with $t<l-2$ and $l>3$. Consider the partial geometry induced by these Latin squares and the point graph of this partial geometry. Then we have $d=t+2<l$, $\lambda_2=l-t-2>0$ and $\lambda_n=-t-2$. Then by Corollary 1.3 and by the fact that $2\sqrt{(t+2)(l-1)}>2(t+1)>t=|\lambda_2-(l-2)|$, the mixing rate of the cliquewise non-backtracking random walk equals
$$
\tilde{\rho}=\frac{1}{\sqrt{(t+1)(l-1)}}.
$$
The mixing rate $\rho'$ of the non-backtracking random walk will be 
\begin{equation}\label{mixingrateoforthogLatinsq}
\rho'=\begin{cases}\frac{l-(t+2)+\sqrt{l^2-6lt-12l+(t+4)^2}}{2\left((t+2)(l-1)-1\right)}&\mbox{if}~l\ge 3(t+2)+2\sqrt{(2t+5)(t+1)}\\
\frac{1}{\sqrt{(t+2)(l-1)-1}} &\mbox{otherwise.}
\end{cases}
\end{equation}
To see this, firstly note that if $\lambda_2\ge |\lambda_n|$, then the fact $l\ge 3(t+2)+2\sqrt{(2t+5)(t+1)}$ is equivalent to $\lambda_2\ge 2\sqrt{d(l-1)-1}$. On the other hand, if $\lambda_2\le |\lambda_n|$, then we have $|\lambda_n|=t+2\le 2\sqrt{(t+2)(t+2)-1}\le 2\sqrt{d(l-1)-1}$. Therefore we have if $t+2<l\le 3(t+2)+2\sqrt{(2t+5)(t+1)}$, 
\begin{eqnarray*}
\tilde{\rho}=\frac{1}{\sqrt{(t+1)(l-1)}}\ge \frac{1}{\sqrt{(t+1)(l-1)+(l-1)-1}}=\rho'
\end{eqnarray*}
The case that $l\ge 3(t+2)+2\sqrt{(2t+5)(t+1)}$ will be a bit more complicated. But we have that 
\begin{eqnarray*}
\frac{\tilde{\rho}}{\rho'}&\le& \frac{2(t+2)(l-1)-1}{\sqrt{(t+1)(l-1)}}\cdot \frac{1}{l-(t+2)}\le \frac{2(t+2)}{\sqrt{t+1}}\cdot\frac{\sqrt{l-1}}{l-(t-2)}\\
&\le &\frac{2(t+2)}{\sqrt{t+1}}\cdot\frac{\sqrt{l}}{l/2}~~(\mbox{by~the~fact~that~}l\ge 2(t+2)\Rightarrow l-(t+2)\ge l/2)\\
&=&\frac{4(t+2)}{\sqrt{t+1}}\cdot\frac{1}{\sqrt{l}}
\end{eqnarray*}
So a sufficient condition for $\tilde{\rho}<\rho'$ is that $l\ge 16(t+2)^2/(t+1)$. Moreover, we have, by (\ref{mixingrateoforthogLatinsq}), 
$$
\frac{\tilde{\rho}}{\rho'}\sim \frac{t+2}{\sqrt{t+1}\sqrt{l}}
$$
Moreover if $t=1$, we obtain Corollary \ref{cor:Latinsquare}.

\section*{Appendix}

Before proving Lemma \ref{mainlemma}, we firstly state a well known result:

\begin{lemma} \label{lem:Harmonic}
Define $\mathbb{T}$ as the unit circle in $\mathbb{C}$. If $\theta/\pi\notin\mathbb{Q}$, then the range of $e^{ik\theta}:\mathbb{Z}\rightarrow\mathbb{T}$ is dense on $\mathbb{T}$. 
\end{lemma}


\begin{proof}[Proof of Lemma \ref{mainlemma}]
If $y\in [-1,1]$, then $y=\cos\theta$ for some $\theta\in[0,\pi]$. We have
\begin{eqnarray*}
|q_k(y)|&=&\Big|\sqrt{(l-1)(1-\delta)(d-1+\delta)}U_k(y)+(l-2)(1-\delta)U_{k-1}(y)\\
&&-\frac{(1-\delta)\sqrt{(l-1)(1-\delta)}}{\sqrt{d-1+\delta}}U_{k-2}(y)\Big|\\
&=& \Big|\sqrt{(l-1)(1-\delta)(d-1+\delta)}\frac{\sin(k+1)\theta}{\sin\theta}+(l-2)(1-\delta)\frac{\sin k\theta}{\sin\theta}\\
&&-\frac{(1-\delta)\sqrt{(l-1)(1-\delta)}}{\sqrt{d-1+\delta}}\frac{\sin(k-1)\theta}{\sin\theta}\Big|\\
&\le & \sqrt{(l-1)(1-\delta)(d-1+\delta)}(k+1)+(l-2)(1-\delta)k\\
&&+\frac{(1-\delta)\sqrt{(l-1)(1-\delta)}}{\sqrt{d-1+\delta}}(k-1)
\end{eqnarray*}
Therefore $\limsup_{k\rightarrow\infty}|q_k(y)|^{1/k}\le 1$. 

Next, if $\theta/\pi\in [0,1/2]\setminus\mathbb{Q}$, take a sequence of integers $k_j$ such that $\lim_{j\rightarrow\infty}\sin(k_j+1)\theta=1$ by Lemma \ref{lem:Harmonic}. So we have $\lim_{j\rightarrow\infty}\sin k_j\theta=\cos\theta> 0$ by the fact that $e^{i(k_j+1)\theta}\rightarrow i$, which yields
\begin{eqnarray*}
&&\lim_{j\rightarrow\infty}q_{k_j}(y)\\
&=&\lim_{j\rightarrow\infty}\Big[\sqrt{(l-1)(1-\delta)(d-1+\delta)}\frac{\sin(k_j+1)\theta}{\sin\theta}+(l-2)(1-\delta)\frac{\sin k_j\theta}{\sin\theta}\\
&&-\frac{(1-\delta)\sqrt{(l-1)(1-\delta)}}{\sqrt{d-1+\delta}}\frac{\sin(k_j-1)\theta}{\sin\theta}\Big]\\
&\ge &\frac{1}{\sin\theta}\Big(\sqrt{(l-1)(1-\delta)(d-1+\delta)}-\frac{(1-\delta)\sqrt{(l-1)(1-\delta)}}{\sqrt{d-1+\delta}}\Big)\\
&\ge & \frac{1}{\sin\theta}\Big(\sqrt{(l-1)(1-\delta)(d-1+\delta)}-\frac{\sqrt{(l-1)(1-\delta)}}{\sqrt{d-1+\delta}}\Big)\\
&=&\frac{1}{\sin\theta}\frac{\sqrt{(l-1)(1-\delta)}}{\sqrt{d-1+\delta}}(d+\delta-2)>0
\end{eqnarray*}
for $d>2$. On the other hand, if $d=2$ and $\delta=0$, we have $l(1-\delta)>d$ by the assumption of the lemma, therefore  $l>d=2$, which yields
\begin{eqnarray*}
&&\lim_{j\rightarrow\infty}q_{k_j}(y)\\
&=&\lim_{j\rightarrow\infty}\Big[\sqrt{(l-1)(1-\delta)(d-1+\delta)}\frac{\sin(k_j+1)\theta}{\sin\theta}+(l-2)(1-\delta)\frac{\sin k_j\theta}{\sin\theta}\\
&&-\frac{(1-\delta)\sqrt{(l-1)(1-\delta)}}{\sqrt{d-1+\delta}}\frac{\sin(k_j-1)\theta}{\sin\theta}\Big]\\
&\ge  &\frac{1}{\sin\theta}\Big(\sqrt{(l-1)(1-\delta)(d-1+\delta)}-\frac{(1-\delta)\sqrt{(l-1)(1-\delta)}}{\sqrt{d-1+\delta}}\Big)\\
&&+(l-2)(1-\delta)\frac{\cos \theta}{\sin\theta}\\
&= &(l-2)\frac{\cos \theta}{\sin\theta}>0
\end{eqnarray*}
Thus $\limsup_{k\rightarrow\infty}|q_k(y)|^{1/k}=1$ for $\theta/\pi\in[0,1/2]\setminus\mathbb{Q}$. 

If $\theta/\pi\in[1/2,3/4]\setminus\mathbb{Q}$, further if $\sqrt{(l-1)(1-\delta)(d-1+\delta)}>(l-2)(1-\delta)/\sqrt{2}$, then take a sequence of integers $k_j$ such that $\lim_{j\rightarrow\infty}\sin(k_j+1)\theta=1$. Thus
\begin{eqnarray*}
&&\lim_{j\rightarrow\infty}\sin k_j\theta\ge -1/\sqrt{2}\\
&&\lim_{j\rightarrow\infty}\sin(k_j-1)\theta\le 0
\end{eqnarray*}
We have
$$
\lim_{j\rightarrow\infty}q_{k_j}(y)\ge \frac{1}{\sin\theta}(\sqrt{(l-1)(1-\delta)(d-1+\delta)}-(l-2)(1-\delta)/\sqrt{2})>0
$$
Thus $\lim_{j\rightarrow\infty}|q_{k_j}(y)|^{1/k_j}$=1. Further if $\sqrt{(l-1)(1-\delta)(d-1+\delta)}\le (l-2)(1-\delta)/\sqrt{2}$, take a sequence of integers $k_j$ such that $\lim_{j\rightarrow\infty}\sin k_j\theta=1$. Thus
\begin{eqnarray*}
&&\lim_{j\rightarrow\infty}\sin(k_j+1)\theta\ge -1/\sqrt{2}\\
&&\lim_{j\rightarrow\infty}\sin(k_j-1)\theta\le 0
\end{eqnarray*}
Then we have
\begin{eqnarray*}
\lim_{j\rightarrow\infty}q_{k_j}(y)&\ge & \frac{1}{\sin\theta}((l-2)(1-\delta)-\sqrt{(l-1)(1-\delta)(d-1+\delta)}/\sqrt{2})\\
&=&\frac{\sqrt{2}}{\sin\theta}((l-2)(1-\delta)/\sqrt{2}-\sqrt{(l-1)(1-\delta)(d-1+\delta)}/2)\\
&>& \frac{\sqrt{2}}{\sin\theta}((l-2)(1-\delta)/\sqrt{2}-\sqrt{(l-1)(1-\delta)(d-1+\delta)})\ge 0
\end{eqnarray*}
by the fact that $l>1$, $\delta<1$ and $d+\delta>1$. So we have $\lim_{j\rightarrow\infty}|q_{k_j}(y)|^{1/k_j}=1$. In summary, we have $\limsup_{k\rightarrow\infty}|q_k(y)|^{1/k}=1$ for $\theta/\pi\in [1/2,3/4]\setminus\mathbb{Q}$. 

If $\theta/\pi\in [3/4,1]\setminus\mathbb{Q}$, further if 
$$
\sqrt{(l-1)(1-\delta)(d-1+\delta)}\ge (l-2)(1-\delta)+\frac{(1-\delta)\sqrt{(l-1)(1-\delta)}}{\sqrt{d-1+\delta}}
$$
take a sequence of integers $k_j$ such that $\lim_{j\rightarrow\infty}\sin(k_j+1)\theta=1$. Thus
\begin{eqnarray*}
&&\lim_{j\rightarrow\infty}\sin k_j\theta=\cos\theta>-1\\
&&\lim_{j\rightarrow\infty}\sin (k_j-1)\theta<1
\end{eqnarray*}
therefore,
\begin{eqnarray*}
&&\lim_{j\rightarrow\infty}q_{k_j}(y)\\
&>&\frac{1}{\sin\theta}\Big(\sqrt{(l-1)(1-\delta)(d-1+\delta)}+ (l-2)(1-\delta)\cos\theta-\frac{(1-\delta)\sqrt{(l-1)(1-\delta)}}{\sqrt{d-1+\delta}}\Big)\\
&\ge &\frac{1}{\sin\theta}\Big(\sqrt{(l-1)(1-\delta)(d-1+\delta)}- (l-2)(1-\delta)-\frac{(1-\delta)\sqrt{(l-1)(1-\delta)}}{\sqrt{d-1+\delta}}\Big)\ge 0
\end{eqnarray*}
where the first proper inequality is by the fact that $\lim_{j\rightarrow\infty}\sin (k_j-1)\theta<1$ and $(1-\delta)\sqrt{(l-1)(1-\delta)}>0$. Thus $\lim_{j\rightarrow\infty}|q_{k_j}|^{1/k_j}=1$. Further if
$$
\sqrt{(l-1)(1-\delta)(d-1+\delta)}< (l-2)(1-\delta)+\frac{(1-\delta)\sqrt{(l-1)(1-\delta)}}{\sqrt{d-1+\delta}}
$$
then take a sequence of integers $k_j$ such that $\lim_{j\rightarrow\infty}\sin k_j\theta=1$. Thus
\begin{eqnarray*}
&&\lim_{j\rightarrow\infty}\sin(k_j+1)\theta=\cos\theta>-1\\
&&\lim_{j\rightarrow\infty}\sin(k_j-1)\theta=\cos\theta>-1
\end{eqnarray*}
Therefore,
\begin{eqnarray*}
&&\lim_{j\rightarrow\infty}q_{k_j}(y)=\frac{1}{\sin\theta}\bigg[ (l-2)(1-\delta)\\
&&+\Big(\sqrt{(l-1)(1-\delta)(d-1+\delta)}-\frac{(1-\delta)\sqrt{(l-1)(1-\delta)}}{\sqrt{d-1+\delta}}\Big)\cos\theta\bigg]\\
&\ge &\frac{1}{\sin\theta}\bigg( (l-2)(1-\delta)-\sqrt{(l-1)(1-\delta)(d-1+\delta)}+\frac{(1-\delta)\sqrt{(l-1)(1-\delta)}}{\sqrt{d-1+\delta}}\bigg)\\
&> &0
\end{eqnarray*}
where the first inequality is by the fact that
\begin{eqnarray*}
&&\Big(\sqrt{(l-1)(1-\delta)(d-1+\delta)}-\frac{(1-\delta)\sqrt{(l-1)(1-\delta)}}{\sqrt{d-1+\delta}}\Big)\cos\theta\\
&\ge &-\Big(\sqrt{(l-1)(1-\delta)(d-1+\delta)}-\frac{(1-\delta)\sqrt{(l-1)(1-\delta)}}{\sqrt{d-1+\delta}}\Big)
\end{eqnarray*} 
Therefore we have $\limsup_{k\rightarrow\infty}|q_k(y)|^{1/k}=1$ for $\theta/\pi\in[3/4,1]\setminus\mathbb{Q}$. 

Next, if $\theta/\pi\in(0,1)\cap\mathbb{Q}$, then the range of $q_k(y)$ is finite for $k\in\mathbb{N}$. We can take a sequence of integers $k_j\nearrow\infty$ such that $\sin k_j\theta=0$, therefore
\begin{eqnarray*}
&&\sin(k_j+1)\theta=\sin\theta\\
&&\sin(k_j-1)\theta=-\sin\theta
\end{eqnarray*} 
and
$$
q_{k_j}(y)=\sqrt{(l-1)(1-\delta)(d-1+\delta)}+\frac{(1-\delta)\sqrt{(l-1)(1-\delta)}}{\sqrt{d-1+\delta}}>0
$$
therefore 
$$
\lim_{j\rightarrow\infty}|q_{k_j}(y)|^{1/k_j}=1
$$
Moreover if $\theta=0$ or $\pi$, then $\cos\theta=\pm 1$,  we have
\begin{eqnarray*}
q_k(1)&=&(k+1)\sqrt{(l-1)(1-\delta)(d-1+\delta)}+k(l-2)(1-\delta)\\
&&-(k-1)\frac{(1-\delta)\sqrt{(l-1)(1-\delta)}}{\sqrt{d-1+\delta}}\\
&>&(k-1)\sqrt{(l-1)(1-\delta)(d-1+\delta)}-(k-1)\frac{\sqrt{(l-1)(1-\delta)}}{\sqrt{d-1+\delta}}\\
&=&(k-1)\frac{\sqrt{(l-1)(1-\delta)}}{\sqrt{d-1+\delta}}(d+\delta-2)\ge 0
\end{eqnarray*}
On the other hand, it is easy to verify by induction that $U_k(-y)=(-1)^kU_k(y)$, thus we have
\begin{eqnarray*}
|q_k(-1)|&=&\Big|(k+1)\sqrt{(l-1)(1-\delta)(d-1+\delta)}-k(l-2)(1-\delta)\\
&&-(k-1)\frac{(1-\delta)\sqrt{(l-1)(1-\delta)}}{\sqrt{d-1+\delta}}\Big|\\
\end{eqnarray*}
which is the absolute value of a linear function of $k$, the cardinality of its range is either 1 or infinity. If the cardinality is infinity, then $|q_k(-1)|>0$ for $k$ large enough. If the cardinality is 1, then we have
$$
\sqrt{(l-1)(1-\delta)(d-1+\delta)}=(l-2)(1-\delta)+\frac{(1-\delta)\sqrt{(l-1)(1-\delta)}}{\sqrt{d-1+\delta}}
$$
which yields
\begin{eqnarray*}
|q_k(-1)|\equiv \sqrt{(l-1)(1-\delta)(d-1+\delta)}+\frac{(1-\delta)\sqrt{(l-1)(1-\delta)}}{\sqrt{d-1+\delta}}>0
\end{eqnarray*}
In summary, we have $\limsup_{k\rightarrow\infty}|q_k(y)|^{1/k}=1$ for $y\in[-1,1]$. 

Next, if $|y|>1$ (same method, see \cite{ABLS}), we have $y=(z+z^{-1})/2$ with $z:=y+\mbox{sign}(y)\sqrt{y^2-1}\notin[-1,1]$. Setting $z:=\mbox{sign}(y)e^\theta$ for some real $\theta$, we have $y=\mbox{sign}(y)\cos(i\theta)$. Therefore if $y>0$ we have that
\begin{eqnarray*}
q_k(y)&=&\sqrt{(l-1)(1-\delta)(d-1+\delta)}\frac{\sin(k+1)i\theta}{\sin i\theta}+(l-2)(1-\delta)\frac{\sin ki\theta}{\sin i\theta}\\
&&-\frac{(1-\delta)\sqrt{(l-1)(1-\delta)}}{\sqrt{d-1+\delta}}\frac{\sin(k-1)i\theta}{\sin i\theta}\\
&=&\sqrt{(l-1)(1-\delta)(d-1+\delta)}\frac{z^{k+1}-z^{-(k+1)}}{z-z^{-1}}+(l-2)(1-\delta)\frac{z^k-z^{-k}}{z-z^{-1}}\\
&&-\frac{(1-\delta)\sqrt{(l-1)(1-\delta)}}{\sqrt{d-1+\delta}}\frac{z^{k-1}-z^{-(k-1)}}{z-z^{-1}}\\
&=:&A(z)z^k+B(z)z^{-k}
\end{eqnarray*}
where $A(z)$ and $B(z)$ do not depend on $k$. We claim $A(z)>0$, to see this,
\begin{eqnarray*}
A(z)&:=& \frac{1}{z^2-1}\Big[\sqrt{(l-1)(1-\delta)(d-1+\delta)}z^2+(l-2)(1-\delta)z\\
&&-\frac{(1-\delta)\sqrt{(l-1)(1-\delta)}}{\sqrt{d-1+\delta}}\Big]\\
&> &\frac{1}{z^2-1}\Big[\sqrt{(l-1)(1-\delta)(d-1+\delta)}+(l-2)(1-\delta)\\
&&-\frac{(1-\delta)\sqrt{(l-1)(1-\delta)}}{\sqrt{d-1+\delta}}\Big]~~~~~~~(\mbox{by~the~fact~that~}z>1)\\
&\ge &0
\end{eqnarray*}
Therefore,
$$
|A(z)||z|^k-|B(z)|\le |q_k(y)|\le |A(z)||z|^k+|B(z)|
$$
Thus $|q_k(y)|^{1/k}\rightarrow |z|=y+\sqrt{y^2-1}$ as $k\rightarrow\infty$. Further if $y<0$, then $z<0$. By the fact that $U_k(-y)=(-1)^kU(y)$, we have
\begin{eqnarray*}
|q_k(y)|&=&\Big|\sqrt{(l-1)(1-\delta)(d-1+\delta)}\frac{\sin(k+1)i\theta}{\sin i\theta}-(l-2)(1-\delta)\frac{\sin ki\theta}{\sin i\theta}\\
&&-\frac{(1-\delta)\sqrt{(l-1)(1-\delta)}}{\sqrt{d-1+\delta}}\frac{\sin(k-1)i\theta}{\sin i\theta}\Big|\\
&=&\Big|\sqrt{(l-1)(1-\delta)(d-1+\delta)}\frac{|z|^{k+1}-|z|^{-(k+1)}}{|z|-|z|^{-1}}-(l-2)(1-\delta)\frac{|z|^k-|z|^{-k}}{|z|-|z|^{-1}}\\
&&-\frac{(1-\delta)\sqrt{(l-1)(1-\delta)}}{\sqrt{d-1+\delta}}\frac{|z|^{k-1}-|z|^{-(k-1)}}{|z|-|z|^{-1}}\Big|\\
&=:&\Big|A(|z|)|z|^k+B(|z|)|z|^{-k}\Big|
\end{eqnarray*}
where 
\begin{equation}\label{eq:A(z)}
A(|z|):=\frac{1}{|z|^2-1}\Big(\sqrt{(l-1)(1-\delta)(d-1+\delta)}|z|^2-(l-2)(1-\delta)|z|-\frac{(1-\delta)\sqrt{(l-1)(1-\delta)}}{\sqrt{d-1+\delta}}\Big)
\end{equation}
and
\begin{equation}\label{eq:B(z)}
B(|z|):=A(|z|^{-1})
\end{equation}
The numerator of $A(|z|)$ has a root $|z|=\sqrt{(l-1)(1-\delta)/(d-1+\delta)}$. However if $l(1-\delta)\le d$, $\sqrt{(l-1)(1-\delta)/(d-1+\delta)}\le 1$, contradicts to the fact that $z\notin[-1,1]$. Therefore we always have $A(|z|)>0$. So by the similar argument as $y>0$, we have $|q_k(y)|^{1/k}\rightarrow |z|$, which provides (\ref{eq:MainLemmaLess}).

For the case that $l(1-\delta)>d$, then $A(|z|)$ in (\ref{eq:A(z)}) for $y<-1$ (hence $z<-1$) has a real root and is $|z_0|=\sqrt{(l-1)(1-\delta)/(d-1+\delta)}>1$. Then $A(|z_0|)=0$ and $B(|z_0|)=A(|z_0|^{-1})\neq 0$, therefore if we assume 
$$
y_0:=-(|z_0|+|z_0|^{-1})/2=\frac{-d-(l-2)(1-\delta)}{2\sqrt{(l-1)(1-\delta)(d-1+\delta)}}
$$
We have
$$
\lim_{k\rightarrow\infty}|q_k(y_0)|^{1/k}=|z_0|^{-1}=\sqrt{\frac{d-1+\delta}{(l-1)(1-\delta)}}
$$ 
which provides (\ref{eq:MainLemmaGreater}).
\end{proof}

\section*{Acknowledgment}

The authors are grateful to the referee for several useful comments and suggestions that have greatly improved the original manuscript. Peng Xu would like to place his sincere gratitude to Dr. Mokshay Madiman for his constant encouragement and valuable help.

%
%
%

\bibliographystyle{amsalpha}

\end{document}